\documentclass[12pt,reqno]{amsart}
 \usepackage{graphicx}
 \usepackage[utf8]{inputenc}
 \usepackage[T1]{fontenc}
 \usepackage{lmodern}
 \usepackage[normalem]{ulem}
 \usepackage{verbatim}
 \usepackage{bbm}
 \usepackage{stmaryrd}
 \usepackage{amsmath}
 \usepackage{amssymb}
 \usepackage{dsfont}
\usepackage{amsthm}

\def\diag{{\rm diag}}

\usepackage{amsmath}%
\usepackage{amsfonts}%
\usepackage{amssymb}%
\usepackage{graphicx}
\usepackage{dsfont}
\newtheorem{theorem}{Theorem}

\newtheorem{assum}[theorem]{Assumption}

\newtheorem{lemma}[theorem]{Lemma}

\newtheorem{prop}[theorem]{Proposition}
\newtheorem{remark}[theorem]{Remark}

\def\tr{{\rm Tr}}

\def\ra{{\rightarrow}}
\def\bG{{\bf G}}
\def\bW{{\bf W}}
\newcommand{\E}{\mathbb E}

\newcommand{\Pp}{\mathbb P}
\newcommand{\C}{\mathbb C}
\newcommand{\R}{\mathbb R}
\newcommand{\N}{\mathbb N}

\newcommand{\Ss}{\mathbb S}

\newcommand\Car{\mathds{1}}
\renewcommand\phi\varphi 

\def\bX{{\bf X}}
\def\bY{{\bf Y}}
\def\bW{{\bf W}}
\usepackage[svgnames]{xcolor}

\newcommand\jonathan[1]{\color{blue} #1 \color{black}}

\begin{document}
\title{Asymptotics of $k$ dimensional spherical integrals and Applications }

\author{Alice~Guionnet}
\address[Alice Guionnet]{Universit\'e de Lyon, ENSL, CNRS,  France}
\email{Alice.Guionnet@ens-lyon.fr}

\author{Jonathan~Husson}
\address[Jonathan Husson]{Universit\'e de Lyon, ENSL, CNRS,  France}
\email{Jonathan.Husson@ens-lyon.fr}
\thanks{This work was supported in part by ERC Project LDRAM : ERC-2019-ADG Project 884584}

\maketitle
\centerline{\bf Abstract}

In this article, we prove that $k$-dimensional spherical integrals are asymptotically equivalent to the product of $1$-dimensional spherical integrals. This allows us to generalize several large deviations principles in random matrix theory known before only in a one-dimensional case. As examples, 
 we study the universality of the large deviations for $k$ extreme eigenvalues of Wigner matrices (resp. Wishart matrices, resp. matrices with general  variance profiles) with sharp sub-Gaussian entries, as well as large deviations principles  for  extreme eigenvalues of Gaussian  Wigner and Wishart matrices with a finite dimensional perturbation.

\section{Introduction}
Spherical integrals are integrals over the unitary or orthogonal group which can be seen as natural Fourier (or Laplace transforms) over matrices. As such, they play a central role in random matrix theory. They can for instance be used to express the density of the distribution of random matrices \cite{CSZ, Zub}. In the unitary case (and more generally when one integrates over  a compact, connected, semisimple Lie group), Harish-Chandra \cite{Harish} and Itzykson and Zuber \cite{itzyksonzuber} derived formulas for such integrals. However, these formulas do not allow to estimate in general  their asymptotics as the dimension goes to infinity because they are given in terms of a determinant, so a signed sum of diverging terms. It is however crucial to estimate such asymptotics in random matrix theory to derive law of large numbers for matrix models or  large deviations principles.  Such asymptotics also permit to see the $R$-transform as the limit of spherical transforms, and thus of natural Laplace transform in the space of matrices \cite{GuMa05}. Such natural representation was recently generalized to the $S$-transform \cite{pot}.  In the case of a one dimensional spherical integral where one integrates over a uniformly distributed vector on the sphere, such asymptotics where derived in \cite{GuMa05} (see also \cite{GoPa}). The case where the spherical integral is taken over the whole unitary or orthogonal group was adressed in \cite{GZ3}. In the case where the exponent is small enough, and the spherical integral is $k$ dimensional, with $k$ much smaller of the dimension, the spherical integrals was shown to be equivalent to a product of one dimensional spherical integrals when $k$ is finite \cite{GuMa05}, or going to infinity in a mesoscopic regime where $k$ grows like a power of the dimension  \cite{CoSn,humeso}. In this article, we show that when $k$ is finite, this property remains true for all ranges of parameters. Indeed, we prove that the limit of $k$-dimensional spherical integrals is equivalent to the sum of one dimensional integrals which are evaluated at the successive largest eigenvalues. For instance, as foreseen in \cite{MaNaPe}, the limit of a 2-dimensional spherical integral depends on the two largest outliers in the large parameters regime,  and not only the top one. As a consequence, $k$-dimensional spherical integrals  allow us to study the universality of the large deviations for $k$ extreme eigenvalues of Wigner matrices with sharp sub-Gaussian entries, hence generalizing the results of \cite{HuGu}  to finitely many extreme eigenvalues. Similarly, we extend  the universality of large deviations for Wishart matrices \cite{HuGu} and matrices with general variance profile  \cite{Hu} with sharp sub-Gaussian  tails to finitely many extreme eigenvalues.
We also prove 
large deviations principles  for  extreme eigenvalues of Gaussian  Wigner and Wishart matrices with a finite dimensional perturbation. This generalizes the one-dimensional case derived in \cite{Ma07}. The large deviations rate functions of these large deviations principles  simply decompose as a sum of the one dimensional rate functions.

The approach of this paper  differs from the arguments used in \cite{GuMa05} in the one-dimensional case which relied heavily on the representation of the uniform law on the sphere in terms of Gaussian variables. Instead, it is based on considering first spherical integrals of matrices with finitely many different eigenvalues where  the uniform law on the sphere can be easily described by Beta-distribution and where rate functions can be more simply described as maximum over real numbers, see section \ref{dis}. We then generalize our results to matrices with continuous spectrum by density, see section \ref{dif}. Applications to large deviations principles for extreme eigenvalues of random matrices are given in section \ref{appl}.

\section{Statement of the results}\label{statement}
We consider  a $N\times N$ Hermitian matrix $\bX_{N}$  such that the empirical measure of its eigenvalues 
$$\hat\mu_{\bX_{N}}=\frac{1}{N}\sum_{i=1}^{N}\delta_{\lambda_{i}}$$
converges towards a probability measure $\mu$ with support with rightmost point $r_{\mu}$ and leftmost point $l_{\mu}$ which are assumed to be finite.  Let $k,\ell$ be two integer numbers.  Let  $\lambda_{1}^{N}\ge\lambda_{2}^N\ge\cdots \ge\lambda_{k}^{N}\ge r_{\mu}$ be  the  $\ell$ largest outliers  of $\bX_{N}$ counted with multiplicity one, $\lambda_{- 1}^{N}\le\cdots\le \lambda_{ - \ell }^{N}\le l_\mu$  be the smallest outliers of $\bX_{N}$ with multiplicity one (but eventually equal).
Assume that
$$\lim_{N\rightarrow\infty}\lambda_{i}^{N}=\lambda_{i}>r_{\mu} \mbox{ for } i\in [1,k], \lim_{N\rightarrow\infty}\lambda_{-i}^{N}=\lambda_{-i}<l_{\mu} \mbox{ for } i\in [1 ,l]$$
The main result of our paper is the following. Denote by $(e_i
)_{-\ell\le i\le k\atop i\neq 0}$ a family of $k+\ell$ orthonormal eigenvectors following the uniform law on the sphere with radius one, taken with complex coordinates if $\beta=2$ and real coordinates if $\beta=1$.  Then

\begin{prop}\label{SIkdcintro} 
Let $\theta_1\ge \theta_2\ge\cdots \ge \theta_{k} \ge 0\ge\theta_{-\ell}\ge\cdots\ge \theta_{-1}$. Then
$$ \lim_{N\rightarrow\infty} \frac{1}{N} \log \E\left[ \exp \left( \frac{\beta N }{2}\sum_{i=-l \atop i \neq 0}^{k}\theta_{i}\langle e_{i},\bX_{N} e_{i}\rangle
 \right) \right]\qquad\qquad$$
$$\qquad\qquad=\frac{\beta}{2}\left(\sum_{i=1}^{k}
J(\mu,\theta_i,\lambda_{i})+\sum_{i=1}^{\ell}
J(\mu,\theta_{-i},\lambda_{-i})\right)\,.
$$
Here,  $J(\mu, \theta, \lambda)= K\left(\mu,\theta, \lambda, v(\mu,\theta, \lambda)\right)
$
with
$$K(\mu,\theta, \lambda, v)=
\theta\lambda+(v-\lambda)G_\mu(v) -\ln\theta -\int \ln  |v-x|d\mu(x)-1$$
and
$$v(\mu,\theta, \lambda)=\left\lbrace\begin{array}{c l}
\lambda &\mbox{ if }  G_\mu(\lambda)\le \theta,\cr
G_\mu^{-1}(\theta) &\mbox{ if }  G_\mu(\lambda)> \theta.\cr\end{array}\right.$$
$G_\mu$ denotes the Cauchy-Stieltjes transform given, for $z$ outside the support of $\mu$, by  $G_\mu(z)=\int (z-x)^{-1} d\mu(x)$.
\end{prop} @@
As a first application, we generalize the universality of the large deviations of the largest eigenvalue for Wigner matrices with sharp sub-Gaussian tails \cite{HuGu} to the $k$-th extreme eigenvalues.
We consider a Wigner matrix $\bX_{N}$ with entries $\left(\frac{X_{ij}}{\sqrt{N}}\right)_{1\le i,j\le N}$ where $(X_{ij})_{i\le j}$ are  independent centered variables such that
\begin{equation}\label{eq-cov}
\E[|X_{ij}|^{2}]=1, \quad  i<j\mbox{ and } \E[|X_{ii}|^{2}]=2^{1_{\beta=1}}\end{equation}
where $\beta=1$ if the entries are real, and $\beta=2$ if they are complex. 
In the complex case we assume that the real and the imaginary part of $X_{ij}, 1\le i<j\le N,$ are independent. We moreover assume that the  $X_{ij}$ have sharp sub-Gaussian tails in the sense that 
\begin{equation}\label{eq-ssg}\E[ \exp( \Re(a X_{ij})) ]\le \exp(\frac {|a|^{2}\E[|X_{ij}|^{2}]}{2})\end{equation}
where $a$ is any complex number  in the case  where  $\beta=2$ and  any  real number in the case $\beta=1$. We finally define the following concentration assumption.
\begin{assum} \label{ass} We say that $\bX_{N}$ concentrates if the spectral radius of $\bX_N$, $||\bX_N||$, concentrates as well as the empirical measure $\hat \mu_{\bX_N}$ of its eigenvalues in the following sense. 
First, $\|X_{N}\|$ 
is exponentially tight at the scale $N$:
 \begin{equation} \label{concspradius} \lim_{K \to + \infty} \limsup_{N \to+ \infty} \frac{1}{N} \log \Pp\big(||\bX_N|| > K\big) = - \infty. \end{equation}
Moreover, the empirical distribution of the eigenvalues $\hat \mu_{\bX_N}$ concentrates at the scale $N$:
\begin{equation}\label{concspmeasure} \limsup_{N \to +\infty} \frac{1}{N} \log \Pp\left( d( \hat \mu_{\bX_N}, \sigma) > \varepsilon(N) \right) = - \infty, \end{equation}
for some $\varepsilon(N)$ goes to zero as $N$ goes to infinity, where $d$ is a distance compatible with the weak topology and $\sigma$ is the semi-circle law, defined by
$$ \sigma(dx) = \frac{1}{2\pi}\sqrt{4-x^2} \Car_{|x|\leq 2}dx.$$

\end{assum}
In our previous paper \cite{HuGu} we took $\varepsilon(N)=N^{-\kappa}$. This hypothesis was needed to insure the continuity of spherical integrals according to \cite{Ma07}. However, part of the consequences of our new approach is that spherical integrals enjoy better continuity properties, see the Appendix \ref{app}. 
Assumption \ref{ass} is satisfied by all the matrix models we shall consider below (Wigner, Wishart and variance profile) as soon as the entries satisfy log-Sobolev inequality with uniformly bounded constant or are bounded (see \cite{AGZ} and the Appendix in \cite{HuGu}).
Examples of entries satisfying all our hypotheses (including \eqref{eq-ssg}) are Rademacher variables or uniform variables. We prove  the following universality of the large deviations for the extreme eigenvalues of $\bX_N$:
\begin{theorem}\label{thmuniv} Let $\bX_{N}=(\frac{X_{ij}}{\sqrt{N}})_{{i,j}}$ be a $N\times N$ Hermitian matrix where $(X_{i,j})_{i\le j}$
are centered independent entries  satisfying \eqref{eq-cov} and \eqref{eq-ssg}, as well as such that $\bX_{N}$ satisfies  Assumption \ref{ass}.
Let $\lambda_{1}^N\ge \lambda_{2}^N\ge\cdots\ge\lambda_{N}^N$ be the $N$ eigenvalues of $\bX_{N}$. Let $k$ be a  fixed integer number.
Then the law of $\bar\lambda^N=(\lambda_{1}^N,\lambda_{2}^N,\ldots,\lambda_{k}^N, \lambda_{N-k}^N,\lambda_{N-k+1}^N,\ldots,\lambda_{N}^N)$ satisfies a large deviations principle in the scale $N$ and with  good rate function
$I(x_{1},\ldots,x_{k},x_{-k},\ldots,x_{-1})$  which  is infinite unless $\bar x=(x_1,\ldots,x_k,x_{-k},\ldots,x_{-1})$ satisfies
$$x_{1}\ge x_{2}\ge\cdots\ge  x_{k}\ge2\ge -2\ge  x_{-k}\ge x_{-k+1}\ge \cdots \ge x_{-1}$$
and is given by $$I(x_{1},\ldots,x_{k},x_{-k},\ldots,x_{-1})=\frac{\beta}{2}\left( \sum_{i=1}^{k}\int_{2}^{x_{i}}\sqrt{t^{2}-4}dt+ \sum_{i=-1}^{-k}\int_{2}^{-x_{-i}}\sqrt{t^{2}-4}dt\right)\,.$$
\end{theorem}
Observe that this result is well known in the Gaussian case for the case $k=1$, see \cite[Section 2.6.2]{AGZ} and \cite{BGD}.  The case of more general $k$ but Gaussian entries is a straightforward generalization, see e.g.  \cite{Giulio}. The case of sharp sub-Gaussian entries and $k=1$ was proven in \cite[Theorem 1.4 and Theorem 1.5]{HuGu}. 
This result can also be generalized for Wishart matrices. We consider $\bG_{L,M}$ a $L \times M$ random matrix and set $N = L+ M$. We define the Wishart matrix $\bW_{L,M} = \frac{1}{{M}} \bG_{L,M} \bG^*_{L,M}$. 
If  $L/M$ goes to {$\alpha\le 1$}, it is well known that
the spectral measure of $\bW_{M,L}$ converges towards the Pastur Marchenko distribution
$$d\pi_\alpha(x)=\frac{1}{2\pi \alpha x}\sqrt{(\lambda_+-x)(x-\lambda_-)}dx\,.$$
where {$\lambda_{\pm}=(1 \pm \sqrt{\alpha})^2$}. Then we have the following : 
\begin{theorem}\label{thmuniv2} Let $\bG_{L,M}=(X_{ij})_{{1\le i\le L\atop 1\le j\le M}}$ be a $L\times M$ matrix where $(X_{i,j})_{i, j}$
are centered independent entries  satisfying \eqref{eq-cov} and \eqref{eq-ssg}, as well as such that $\bW_{{L,M}}$ satisfies Assumption \ref{ass}.
Let $k \geq 0$ and $\lambda^N_1 \geq ... \geq \lambda^N_k$ the $k$ largest eigenvalues of $\bW_{L,M}$. Assume that there exists {{$\alpha \le1$ }} and $\kappa \geq 0$ so that  $L/M - \alpha = o(N^{- \kappa})$. Then $( \lambda^N_1,..., \lambda^N_k)$ satisfies a large deviations principle in the scale $N$ with good rate function $J(x_1,...,x_k)$ which is infinite unless $x_1 \geq ... \geq x_k \geq b_{\alpha}$ and :
\[ J(x_1,...,x_k)=  \frac{\beta}{{4(1 + \alpha)}} \sum_{i=1}^k \int_{\lambda_-}^{x_i} \frac{\sqrt{(y - \lambda_{-})(y- \lambda_+)}}{y} dy \]
\end{theorem}
As in the Wigner case, as soon as the entries satisfy log-Sobolev inequality or are compactly supported, insuring that the empirical measure converges towards the Pastur-Marchenko distribution with probability larger than any exponential and that the norm of Wishart matrices stays bounded, yielding a property similar to Assumption \ref{ass}  for $\bW_{L,M}$.

This result can be further extended to Wigner matrices with variance profiles. Those matrices are built by letting $\bX^\sigma_N(i,j) = \sigma_N(i,j) \frac{X_{i,j}}{\sqrt{N}}$ where : 
\begin{itemize}
	\item either there exists $p \in \N$, $\alpha_1(N),...,\alpha_p(N)>0$ such that $\sum_1^p \alpha_i(N)=N$ and  $\lim \alpha_i(N)/N = \alpha_i >0 $,  and $(\sigma_{i,j})_{i,j} \in M_{p,p}(\R^+)$, {$\sigma=\sigma^T$}, such that :
	\[ \sigma_N(i,j)=  \sum_{k,l =1}^p \sigma_{k,l} \mathds{1}_{I^k_N \times I^l_N}(i,j) \]
	where $I^1_N = [0, \alpha_1(N)]$ and $I^{i+1}_N = [\sum_{j=1}^i \alpha_j(N) + 1 , \sum_{j=1}^{i+1} \alpha_j(N) ]$. This case will be called the piecewise constant case with parameters $\sigma$ and $\alpha$. 
	\item either $\sigma_N(i,j) = \sigma(i/N,j/N)$ where $\sigma$ is a continuous symetric positive function of $[0,1]^2$. This case will be called the continuous case. 
\end{itemize}
We will also make the following assumption on the variance profiles :
\begin{assum}\label{Neg}
	\begin{itemize}
		\item In the piecewise constant case, we assume that the quadratic form $\psi \mapsto \sum_{i,j}^p \sigma^2_{i,j} \psi_i \psi_j$ is negative on the subspace $Vect(1,...,1)^{\bot}$. 
		\item In the continuous case, we assume that the the function $\psi \mapsto \int \sigma^2(x,y) d \psi(x) d \psi(y)$ is concave on the set $\mathcal{P}([0,1])$ of probability measures on $[0,1]$. 
		\end{itemize}
	\end{assum}
When this Assumption as well as Assumption \ref{concspmeasure}  and \eqref{eq-ssg} are verified, one of the author of this article \cite{Hu} proved that the largest eigenvalue of $\bX_N^\sigma$ satisfies a large deviations principle with a good rate function $J^{(1)}_\sigma$. In this article we generalize this result to the $k$th largest eigenvalues and prove  the following theorem :

\begin{theorem} \label{thmuniv3} Let $\bX^{\sigma}_{N}=(\frac{\sigma_{N}(i,j)X_{ij}}{\sqrt{N}})_{{i,j}}$ be a $N\times N$ Hermitian matrix where $(X_{i,j})_{i\le j}$
are centered independent entries  satisfying \eqref{eq-cov} and \eqref{eq-ssg}, as well as such that $\bX_{N}^{\sigma}$ satisfies  Assumption \ref{ass} {and such that $\sigma$ verifies Assumption \ref{Neg}}.
Let $k \geq 0$ and $\lambda^N_1 \geq ... \geq \lambda^N_k$  be the $k$ largest eigenvalues of $\bX_N^\sigma$. Then $( \lambda^N_1,..., \lambda^N_k)$ satisfies a large deviations principle in the scale $N$ with good rate function $J_\sigma^{(k)}(x_1,...,x_k)$ which is infinite unless $x_1 \geq ... \geq x_k \geq r_{\sigma}$ and in this case equals:
\[ J_\sigma^{(k)}(x_1,...,x_k)=   \sum_{i=1}^k J_\sigma^{(1)}(x_i) \]	
where $r_{\sigma}$ is the rightmost point of the support of the limit $\mu_{\sigma}$ of the empirical measure of $\bX_N$ and where $J^{(1)}_\sigma$ is the rate function for the large deviations of the largest eigenvalue. 
\end{theorem}
	This result was proved in the case $k=1,\ell=0$ in \cite[Theorem 6]{GuMa05}.
Let us now  consider  $\bX_{N}$  to be a GOE/GUE matrix, that is a $N\times N$ Hermitian matrix with centered real/complex Gaussian entries satisfying \eqref{eq-cov}. 
Let $\ell$ and $k$ be two integer numbers  and let   $(e_{1},\ldots e_{k}, e_{-1},\ldots,e_{-\ell})$ be  orthonormal vectors following the uniform law on the sphere. In \cite{Ma07}, Mylene Ma\"{\i}da showed that the largest eigenvalue of a Gaussian Wigner matrix perturbed by a rank one matrix satisfy a large deviations principle. In this article we generalize this result to the $k$th largest eigenvalues and $\ell$ smallest eigenvalues when the Gaussian matrix is perturbed by a finite rank matrix with $k$ non-negative eigenvalues and $\ell$ non-positive eigenvalues. 
\begin{prop}\label{Wigper} Let $\bX_N$ be a GUE ($\beta=2$) or GOE ($\beta=1$) matrix. Let  $\ell,  k$ be  two finite integer numbers.
Let $\theta_1\ge \theta_2\ge\cdots \ge \theta_{k}\ge  0\ge\theta_{-\ell}\ge\cdots\ge \theta_{-1}$
and define
$$\bX_{N}^{\theta}=\bX_{N}+\sum_{-\ell\le i\le k\atop i\neq 0}\theta_{i} e_{i}e_{i}^{*}\,.$$
Let $\lambda_{1}^{N,\theta}\ge \lambda_{2}^{N,\theta}\ge\cdots\ge \lambda_{N}^{N,\theta}$ be the  eigenvalues of $\bX_{N}^{\theta}
$.  Then, the  distribution  of $(\lambda_{1}^{N,\theta}, \ldots,\lambda_{k}^{N,\theta}, \lambda_{N-\ell}^{N,\theta},\ldots,\lambda_{N}^{N,\theta})$ satisfies a large deviations principle in the  scale $N$  and  with good rate function which  is infinite unless $$x_{1}\ge x_{2}\ge\cdots\ge  x_{k}\ge 2\ge -2\ge x_{-\ell}\ge \cdots \ge x_{-1}$$
and is given then by $\beta \sum_{-\ell\le i\le k\atop i\neq 0} I_{\theta_{i}}(x_{i})$.
Here, with  $I(y)=\frac{1}{4} x^{2}-\int\ln|x-y|d\sigma(y)$, we have set
$$I_{\theta}(x)=I(x)-J(\sigma, \theta,x)-\inf_{y}(I(y)-J(\sigma, \theta,y))\,.$$
\end{prop}
A similar result holds for Wishart matrices.  We next consider a $L\times M$ matrix $\bG_{L,M}$ with i.i.d standard Gaussian matrices with covariance $1$, set $N=M+L$, and assume without loss of generality that $M\ge L$.  We consider the Wishart matrix
$$\bW_N^\gamma=\frac{1}{M}\Sigma_L^{1/2}\bG_{L,M}\bG_{L,M}^* \Sigma_L^{1/2}$$
where $\Sigma_L$ is a $L\times L$ covariance matrix given by $I_L+\sum_{i=1}^k \gamma_i e_i e_i^*$ for some fixed $\gamma_i>-1$.  Here the $e_i, 1\le i\le k$ are orthonormal vectors. It is well known that  when $L/M$ goes to $\alpha\in [0,1]$, the empirical measure of the eigenvalues of $\bW_N^\gamma$ goes to the Pastur-Marhenko distribution. 

Large deviations for the extreme eigenvalues in the case $\gamma_i=0$ are well known, and similar to the Gaussian case, see \cite{AGZ,Majum2}. The rate function governing the large deviations {in the scale $N$} for the smallest eigenvalue is infinite outside $[0,\lambda_-]$ and is given for $y\in [0,\lambda_-]$ by
\begin{eqnarray*}I_\alpha(y)&=&\frac{\beta}{4{(1+\alpha)}}( y-(1-\alpha)\ln y-2\alpha\int\ln |y-t|d\pi_\alpha(y)-C)\\
&=&\frac{\beta}{4{(1+\alpha)}} \int^{\lambda_-}_y\frac{1}{ t}\sqrt{(t-\lambda_+)(t-\lambda_-)} dt\end{eqnarray*}
where $C$ is the infimum of $y-(1-\alpha)\ln y-2\alpha\int\ln |y-t|d\pi_\alpha(y)$. The same result holds for the largest eigenvalue. We have the following analogue to Proposition \ref{Wigper}.
\begin{prop}\label{Wisper}Let  $\ell\le k$ be  two finite integer numbers.
Let $\gamma_1\ge \gamma_2\ge\cdots \ge \gamma_{\ell}\ge  0\ge\gamma_{\ell+1}\ge\cdots\ge \gamma_{k}> -1$.
Let $\lambda_{1}^{N,\gamma}\ge \lambda_{2}^{N,\gamma}\ge\cdots\lambda_{M}^{N,\gamma}$ be the  eigenvalues of $\bW_{N}^\gamma
$ in decreasing order.  Then, the  law of $(\lambda_{1}^{N,\gamma}, \ldots,\lambda_{\ell}^{N,\gamma}, \lambda_{N-k+\ell}^{N,\gamma},\ldots,\lambda_{N}^{N,\gamma})$ satisfies a large deviations principle in the  scale $N$  and  with good rate function which  is infinite unless $$x_{1}\ge x_{2}\ge\cdots\ge  x_{\ell}\ge \lambda_+\ge \lambda_-\ge x_{\ell+1}\ge \cdots \ge x_{k}\ge 0$$
and is given otherwise by $ \sum_{i=1}^{k} I_{\gamma_{i},\alpha}(x_{i})$.
Here,  we have set
$$I_{\gamma,\alpha}(x)=I_\alpha(x)-\frac{\beta}{2} J(\pi_\alpha, \frac{\gamma}{1-\gamma},x)-\inf_{y}(I_\alpha(y)-\frac{\beta}{2}J(\pi_\alpha, \frac{\gamma}{1-\gamma},y))$$
\end{prop}

We finally notice that since our results hold for any number of eigenvalues, they capture as well the large deviations for the point processes of the outliers. For instance, if we let $A_{i}=[a_{i},b_{i}]$ be intervals above the bulk, $b_{i}<a_{i+1}<b_{i+1}$, if we denote $I$ the large deviation rate function for
 any of the above models, the probability that there are $n_{i}$ outliers in the set $A_{i}$ has probability of  order $\exp\{-N\sum_{i} n_{i }\inf_{A_{i}} I\}$.

{\bf Acknowledgements} We are very grateful to Myl\`ene Ma\"{\i}da and Marc Potters for preliminary discussions about the  questions addressed in this paper.

\section{Limiting spherical integral in the discrete case}\label{dis}
We  first  consider the case where $\bX_N$ has finitely many different eigenvalues :
$$\bX_N= diag\left( \underbrace{\eta_{-m_{1}+1},...,\eta_{-m_{1}+1}}_{N_{-m_{1}+1}  },\underbrace{\eta_{-m_{1}+2},...,\eta_{-m_{1}+2}}_{N_{-m_{1}+2}  },..., \underbrace{\eta_{p+m_{2}},...,\eta_{p+m_{2}} }_{N_{p+m_{2}} }\right),$$ where $\eta_{-m_{1}+1} < ... < \eta_1\cdots<\eta_p < \eta_{p+1}<\cdots < \eta_{p+m_{2}}$, and $\eta_i$ has multiplicity $N_i$  where $\sum_{i=-m_{1}+1}^{p+m_{2}} N_i=N $. We  assume that $N_{i}/N$ goes to a positive limit  $\alpha_{i}$ for {$i\in \{1, p\}$} and to zero for $i\in\{1-m_{1},\ldots,0\}\cup\{ p+1,\ldots, p+m_{2}\}$, the later representing the outliers of $\bX_{N}$. $m_{1}, m_{2},p$ are independent of $N$ {(with the convention that if $\alpha_i =0$, $\alpha_i \ln \alpha_i =0$)}. 
In the previous notations, the eigenvalues of $\bX_N$ are given by $\lambda_1^N\ge \lambda_2^N\ge \cdots\ge \lambda_N^N$ with 
$$\lambda_i^N=\eta_j, i\in I_j=[N_{p+m_2}+\cdots+N_{p+m_2-j}+1, N_{p+m_2}+\cdots+N_{p+m_2-j}+N_{p+m_2-j -1}]$$
and $I_{p+m_2}=[1,N_{p+m_2}]$. 

\begin{remark}\label{SICont}
	{ We notice that if the sequences $N_i$ are fixed, the spherical integrals are $\beta/2$-Lipschitz in the $p+m_1+m_2$-uplet $(\eta_i)_{-m_1 +1 \leq i \leq p+m_2}$ with the norm $||.||_{\infty}$ }.
	\end{remark}

\subsection{Limiting 1-d spherical integral}
We start by giving a new proof of \cite[Theorem 6]{GuMa05} giving the asymptotics of spherical integrals in the one dimensional case, in the case of matrices with $p+m_1+m_2$ different eigenvalues with multiplicity as above. This proof will in fact extend to the higher dimensional setting in the next subsection. 
\begin{prop}\label{SI1d}
Let $\theta\ge 0$. Then
$$ \lim_{N\rightarrow\infty} \frac{1}{N} \log \E\left[ \exp \left( \frac{\beta N }{2}\theta \langle e , \bX_N e\rangle \right) \right]=\frac{\beta}{2}
\sup_{\gamma_i\ge 0\atop \sum_{i=1-m_1}^{p+m_2}\gamma_i=1}\left\lbrace{\theta\sum_{i=-m_1+1}^{p+m_2} \eta_i \gamma_i+\sum_{i=1}^p \alpha_i \ln\frac{\gamma_i}{\alpha_i}}\right\rbrace
$$
\end{prop}
\begin{proof}
We have the following formula   : 
\[ \langle e, \bX_N e \rangle = \sum_{i=-m_{1}+1}^{p+m_{2}} \eta_i \gamma_i^N  \]
where we have denoted 
$ \gamma_j^N = \sum_{i \in I_j} |u_i|^2$
with $u_i=\langle v_i, e\rangle$ if $v_i$ is the eigenvector for the $i$-th eigenvalue of $\bX_N$. In other words, $\gamma_j^N$ is the $\ell^2$-norm of the projection of $e$ onto the eigenspace of $\eta_j$. 
The vector $\gamma^N$ follows a Dirichlet law of parameters $\frac{\beta}{2} ( N_{1-m_1} , \ldots, N_{p+m_{2}} )$, that is the distribution on 
$\Sigma=\{x\in [0,1]^{m_{1}+m_{2}+p}:\sum_{i=1-m_{1}}^{m_{2}+p}x_{i}=1\}$ given by $\gamma_{1-m_1}^N=1-\sum_{i=2-m_1}^{p+m_2} \gamma_i^N$ and 

\begin{equation}\label{ff}
d\mathbb P^{N}_{\bf N}(\gamma)=\frac{1}{Z_{\alpha}^{N}}{\bf 1}_{\sum_{i=2-m_1}^{p+m_2} \gamma_i\le 1} (1-\sum_{i=2-m_1}^{p+m_2} \gamma_i)^{\frac{\beta}{2}N_{1-m_1} -1}
\prod_{j=2-m_{1}}^{p+m_{2}} \gamma_{j}^{\frac{\beta}{2}N_{i} -1}{\bf 1}_{\gamma_j\ge 0} d\gamma_{j}\end{equation}
We deduce the following large deviations principle
\begin{theorem}\label{LDP1} Assume that $N_i/N$ converges towards $\alpha_i$ for all $i$, with $\alpha_i=0$ for $i\notin [1,p]$.
Then, the law of $\gamma^N$ satisfies a large deviations principle with scale $N$ and good rate function  $I_\alpha$ given for $x\in \Sigma$ by
\[ I_{\alpha}(x_{1-m_{1}},...,x_{p+m_{2}}) =  \frac{\beta}{2} \sum_{i=1}^{p} \alpha_i \log \frac{x_i}{\alpha_{i}} \,.\]
\end{theorem}
The proof is   a direct consequence of Laplace's method. 
We deduce Proposition \ref{SI1d} by Varadhan's lemma. 
\end{proof}
\begin{lemma} \label{formulaI}For $\theta\ge 0$, $\eta=(\eta_{1-m_1}<\cdots <\eta_{p+m_2})$ and $\alpha_i\in(\mathbb R^+)^p$ such that $\sum_{i=1}^p\alpha_i=1$, 
$$J(\theta,\eta)=\sup_{\gamma_i\ge 0\atop \sum\gamma_i=1}\left\lbrace{\theta\sum_{i=-m_1+1}^{p+m_2} \eta_i \gamma_i+\sum_{i=1}^p \alpha_i \ln\frac{\gamma_i}{\alpha_i}}\right\rbrace$$ only depends on $\eta_{p+m_2},\theta$ and $\mu= \sum_{i=1}^p\alpha_i\delta_{\eta_i}$. It is given by
$$J(\theta,\eta)= J(\mu, \theta, \eta_{p+m_2})= K\left(\mu,\theta, \eta_{p+m_2}, v(\mu,\theta, \eta_{p+m_2})\right)
$$
with
$$K(\mu,\theta, \lambda, v)=
\theta\lambda+(v-\lambda)G_\mu(v) -\ln|\theta| -\int \ln  |v-x|d\mu(x)-1$$
and
$$v(\mu,\theta, \lambda)=\left\lbrace\begin{array}{c l}
\lambda &\mbox{ if }  G_\mu(\lambda)\le \theta,\cr
G_\mu^{-1}(\theta) &\mbox{ if }  G_\mu(\lambda)> \theta.\cr\end{array}\right.$$
\end{lemma}
\begin{proof}
$J(\theta,\eta)$ is the  supremum of
\begin{equation}\label{defI}I_{\theta,\eta}^{p+m_2}(\gamma):=
\theta\sum_{i=-m_1+1}^{p+m_2} \eta_i \gamma_i+\sum_{i=1}^p \alpha_i \ln\frac{\gamma_i}{\alpha_i}\,.\end{equation}
The entropic term in $I_{\theta,\eta}^{p+m_2}$
 does not depend on  $(\gamma_i, i<1 \mbox{ or } i>p)$, and the first term increases when we take them all equal to zero except $\gamma_{m_2+p}$. Hence, the maximum is taken at $\gamma_i=0$ for $i<1$ or $i\in [p+1,p+m_2-1]$.
 Then, putting $\gamma_{p+m_2}=1-\sum_{i=1}^p \gamma_i$ we see that we need to optimize

$$I_{\theta,\eta}(\gamma)=\theta\eta_{p+m_2}+\left\lbrace{\theta\sum_{i=1}^{p} (\eta_i -\eta_{p+m_2})\gamma_i+\sum_{i=1}^p \alpha_i \ln\frac{\gamma_i}{\alpha_i}}\right\rbrace$$
over $\gamma_i\ge 0, \sum_{i=1}^p \gamma_1\le 1$. We see that the critical point of $I_{\theta,\eta}(\gamma)$ is
$$\gamma^*_i=\frac{\alpha_i}{\theta (\eta_{p+m_2}-\eta_i)}, 1\le i\le p, \,\gamma_{p+m_{2}}^*=1-\sum_{i=1}^p \gamma_i^*=1-\frac{1}{\theta}G_{\mu}(\eta_{p+m_{2}})$$
provided $\sum_{i=1}^p \gamma_i^*=\frac{1}{\theta}G_\mu(\eta_{p+m_2})\le 1$. 
  For $\theta<G_\mu(\eta_{p+m_2})$, the supremum is achieved at 
$$\gamma^{**}_i=\frac{\alpha_i}{\theta (G_\mu^{-1}(\theta)-\eta_i)},1\le i\le p, \,\gamma_{p+m_{2}}^{**}=0$$
This gives  the announced formula.
\end{proof}

\subsection{Limiting 2-d spherical integral}
We next consider the bi-dimensional case where 
 $(e,f)$ are  two orthonormal vectors following the uniform law in the sphere.
\begin{prop}\label{SI2d}
Let $\theta_1\ge \theta_2\ge 0$. Then, if $N_{p+m_{2}}=1$,
$$ \lim_{N\rightarrow\infty} \frac{1}{N} \log \E\left[ \exp \left( \frac{\beta N }{2}(\theta_1 \langle e , \bX_N e\rangle + \theta_2 \langle f, \bX_N f \rangle ) \right) \right]\qquad\qquad$$
$$\qquad\qquad=\frac{\beta}{2}(J(\mu,\theta_1,\eta_{p+m_2})+J(\mu,\theta_2,\eta_{p+m_2-1}))\,.
$$
If $N_{m+p_{2}}\ge  2$,
$$ \lim_{N\rightarrow\infty} \frac{1}{N} \log \E\left[ \exp \left( \frac{\beta N }{2}(\theta_1 \langle e , \bX_N e\rangle + \theta_2 \langle f, \bX_N f \rangle ) \right) \right]\qquad\qquad$$
$$\qquad\qquad=\frac{\beta}{2}(J(\mu,\theta_1,\eta_{p+m_2})+J(\mu,\theta_2,\eta_{m+p_2}))\,.
$$

\end{prop}

\begin{proof}
We first assume that $N_{p+m_{2}}=1$. We can write : 

$$
\E\left[ \exp \left( N\frac{\beta}{2} \left( \theta_1 \langle e, \bX_N e\rangle + \theta_2 \langle f, \bX_N f \rangle \right) \right)\right] \qquad\qquad$$
$$\qquad \qquad= \E\left[ \exp \left( N\frac{\beta}{2}  \theta_1 \langle e, \bX_N e\rangle\right)\mathbb E\left[\exp\left( \frac{\beta}{2} \theta_2 \langle f, \bX_N^{(e)} f \rangle \right) \bigg| e \right]\right] 
$$
where $\bX_N^{(e)}=P_{e^\bot}\bX_NP_{e^\bot}$ if $P_{e^\bot}=I-ee^*$ is the orthogonal projection onto the ortho-complement of $e$.
 We can see $\bX_N^{(e)}$ as a $(N-1)\times (N-1)$ matrix living in  $Vect(e)^{\bot}$.
Its largest eigenvalue $\chi $ belongs to $[\eta_{p+m_2-1},\eta_{p+m_2}]$ and writing that the corresponding eigenvector $v\in Vect(e)^{\bot}$ satisfies $\bX_N^{(e)}v=\chi  v$ and $\langle e, v\rangle=0$,
we find that $\chi $ belongs to $[\eta_{p+m_2-1},\eta_{p+m_2}]$ must satisfy 
\begin{equation}\label{zx}\sum_{i=1-m_1}^{p+m_2}\frac{\gamma_i(e)}{\chi -\eta_i}=0\end{equation}  if there is  a solution in this interval where $\gamma_i(e)=\sum_{j\in I_i} |\langle v_j,e\rangle|^2$ with $v_j$ the $j$th eigenvector of $\bX_N$. If there is no solution (which can happen only if $\gamma_{p+m_2}(e) =0$) then $\chi = \eta_{p + m_2}$ if the rational function is positive on this interval and $\chi = \eta_{p + m_2 -1}$ if it is negative. 
 Note that $\chi=\chi(\gamma(e))$ is a continuous function of $\gamma(e)$ and denote this function $\chi (\gamma(e))$. Moreover, the spectral measure of $\bX_N^{(e)}$ converges towards $\mu$, the limiting spectral measure of $\bX_N$ by Weyl interlacing property. 
 Therefore, when $\gamma(e)$ converges towards $\kappa$, and since the empirical measure of $\bX_N^{(e)}$ converges toward the same limit that the empirical measure of $\bX_N$,
$$\lim_{N\rightarrow \infty}\frac{1}{N-1} \ln \mathbb E[ e^{\theta_2 N\frac{\beta}{2}\langle f, \bX_N f\rangle}|e]=\frac{ \beta}{2} J(\mu,x, \chi (\kappa)).$$
Moreover, the right hand side depends continuously on $\kappa$ (since $J$ is continuous in $\chi$ and $\chi$ in $\kappa$).
We now can apply the fact that $\gamma^N$ follows a large deviations principle, see Theorem \ref{LDP1}, to conclude that

\begin{multline*}
\lim \frac{1}{N} \log { \E[ \exp \left( N\frac{\beta}{2} \left( \theta_1 \langle e, \bX_N e\rangle + \theta_2 \langle f, \bX_N f \rangle \right) \right) ]} \\=\frac{\beta}{2} \sup_{\gamma\in (\R^{+})^{p+m_1+m_2}, \sum \gamma_i =1} \left(J(\mu,\theta_2, \chi (\gamma)) + \sum_{i=1}^p \alpha_i \log\frac{ \gamma_i }{\alpha_{i}}+\sum_{i=1-m_1}^{p+m_2} \theta_1 \eta_i \gamma_i  \right) 
\end{multline*}
  Since $J$ is bounded and due to the continuity of $\gamma$, we can change the domain of the $\sup$ to $(\R^{+,*})^{p+m_1+m_2}$. 
We next complete the proof by computing the right hand side and showing it equals the sum of the two limiting spherical integrals as stated. We first  denote by  $\tilde \gamma_{i}=\gamma_{i }|\chi -\eta_{i}|^{-1}$ with $\chi =\chi (\gamma)$. By definition we have $\tilde \gamma_{i} > 0$, $\chi \in (\eta_{p+m_2-1},\eta_{p+m_2})$ and \eqref{zx} holds so that

$$\tilde\gamma_{p+m_{2}}=\sum_{i=1-m_{1}}^{p +m_{2}-1}\tilde \gamma_{i}, \qquad \sum_{i=1-m_{1}}^{p+m_{2}-1}(\chi -\eta_{i})\tilde \gamma_{i}+(\eta_{p+m_{2}}-\chi )\tilde\gamma_{p+m_2}=1$$
This simplifies into the condition
\begin{equation}\label{nm}\tilde\gamma_{p+m_{2}}=\sum_{i=1-m_{1}}^{p+m_{2}-1 }\tilde \gamma_{i},\qquad  \eta_{p+m_{2}}\tilde\gamma_{p+m_{2}}- \sum_{i=1-m_{1}}^{p+m_{2}-1}\eta_{i}\tilde \gamma_{i}=1\end{equation}
which is independent of $\chi $. We thus first take the supremum over $\chi \in [\eta_{p+m_2-1},\eta_{p+m_2}]$ of

\begin{eqnarray*}
I(\chi ,\tilde\gamma)&=&J(\mu,\theta_2, \chi ) + \sum_{i=1}^{p}\alpha_i \log [|\eta_{i}-\chi | \frac{\tilde \gamma_i }{\alpha_i} ]-\theta_{1}(\chi -\eta_{p+m_{2}}) \eta_{p+m_{2}}\tilde \gamma_{p+m_{2}}\\
&&+ \sum_{i=1-m_{1}}^{p+m_{2}-1} \theta_1(\chi -\eta_{i}) \eta_i \tilde \gamma_i \\
&=& H(\chi ) + \sum_{i=1}^{p}\alpha_i \log \frac{\tilde \gamma_i }{\alpha_i} +\theta_{1}\eta_{p+m_{2} }^{2}\tilde\gamma_{p+m_{2}}- \sum_{i=1-m_{1}}^{p+m_{2}-1} \theta_1 \eta_i^{2} \tilde \gamma_i  \\
\end{eqnarray*}
with
\begin{equation}\label{defH}H(\chi )= J(\mu,\theta_2, \chi ) +\sum_{i=1}^{p}\alpha_i \log |\eta_{i}-\chi | -\chi \theta_{1}.\end{equation}
Recall the formula for $J$ from Lemma \ref{formulaI}.
When $\theta_{2}\le G_{\mu}(\chi )$, that is $\chi\le G_\mu^{-1}(\theta_2)$,
$J$ does not depend on  $\chi $ and the function $H$ increases till $G_{\mu}^{-1}(\theta_{1})$ and decreases afterwards. 
When $\theta_{2}\ge G_{\mu}
(\chi )$, that is $\chi\ge G^{-1}(\theta_2)$, Lemma \ref{formulaI} gives

\begin{eqnarray*}
H(\chi )&=&\theta_{2}\chi  -\sum_{i=1}^{p}\alpha_{i}\ln(\chi -\eta_{i})-\ln \theta_{2}-1+\sum_{i=1}^{p}\alpha_i \log |\eta_{i}-\chi |-\chi \theta_{1}\\
&=&\chi (\theta_{2}-\theta_{1})-\ln \theta_{2}-1\end{eqnarray*}
which is  decreasing  since $\theta_{1}>\theta_{2}$. Therefore,  $H$ increases till $G_{\mu}^{-1}(\theta_{1})$ and decreases afterwards. As a consequence,
\begin{equation}\label{maxH}\max_{\chi\in [\eta_{p+m_2-1},\eta_{p+m_2}]}
H(\chi)=\left\lbrace
\begin{array}{ll}
H(\eta_{p+m_2-1})& \mbox{ if }  G_\mu^{-1}(\theta_1)\le \eta_{p+m_2-1},\cr
H( G_{\mu}^{-1}(\theta_{1}))&\mbox{ if }  G_\mu^{-1}(\theta_1)\in [\eta_{p+m_2-1},\eta_{m+p_2}],\cr
H(\eta_{p+m_2}) &\mbox{ if }  G_\mu^{-1}(\theta_1)>  \eta_{p+m_2}.\cr\end{array}\right. \end{equation}
Let us also  optimize on $\tilde \gamma$ satisfying \eqref{nm} the function

$$L(\tilde\gamma)= \sum_{i=1}^{p}\alpha_i \log [\frac{\tilde \gamma_i }{\alpha_i}]+\theta_{1}\eta_{p+m_{2} }^{2}\tilde\gamma_{p+m_2}- \sum_{i=1-m_{1}}^{p+m_{2}-1} \theta_1 \eta_i^{2} \tilde \gamma_i \,.$$
Replacing $\tilde\gamma_{p+m_{2}}$ by  $\sum_{i=1-m_{1}}^{p+m_{2}-1}\tilde\gamma_{i}$ we get
$$L(\tilde\gamma)=\sum_{i=1}^{p}\alpha_i \log [\frac{\tilde \gamma_i }{\alpha_i}]+\theta_{1} \sum_{i=1-m_1}^{p+m_2-1} (\eta_{p+m_{2}}^{2} -\eta_i^{2} )\tilde \gamma_i $$
  with by \eqref{nm},  $\sum (\eta_{p+m_{2}}-\eta_{i})\tilde\gamma_{i}=1$. We may  again do the change of variables $\bar\gamma_{i}=(\eta_{p+m_{2}}-\eta_{i})\tilde\gamma_{i}$ which are non negative and with mass one by \eqref{nm}.
  We get by \eqref{nm}
  \begin{eqnarray*}L(\tilde\gamma)&=&\sum_{i=1}^{p}\alpha_i \log [\frac{\bar \gamma_i }{\alpha_i(\eta_{p+m_{2}}-\eta_{i})}]+\theta_{1} \sum_{i=1-m_{1}}^{p+m_{2}-1} (\eta_{p+m_2}+\eta_i )\bar \gamma_i \\
  &=&
  \sum_{i=1}^{p}\alpha_i \log [\frac{\bar \gamma_i }{\alpha_i(\eta_{p+m_2}-\eta_{i})}]+\theta_{1} \sum_{i=1-m_{1}}^{p+m_{2}-1} \eta_i \bar \gamma_i +\theta_{1}\eta_{p+m_{2}}\\
  &=& I_{\theta_1,\eta}^{p+m_2-1}(\bar\gamma)+  \sum_{i=1}^{p}\alpha_i \log [\frac{1}{(\eta_{p+m_2}-\eta_{i})}]+\theta_{1}\eta_{p+m_{2}}
   \end{eqnarray*}
   where 
   $I_{\theta_1,\eta}^{p+m_2-1}$ is defined as in \eqref{defI} with largest outlier $\eta_{p+m_2-1}$. Its maximum gives $J(\mu,\theta_1,\eta_{p+m_2-1})$. 
    We thus get
  $$\max L=
  \sum_{i=1}^{p}\alpha_i \log [\frac{1}{(\eta_{p+m_2}-\eta_{i})}]+\theta_{1}\eta_{p+m_{2}}+J(\mu,\theta_{1},\eta_{p+m_{2}-1})$$
  We finally compute 
  $\max I(\chi,\tilde\gamma)= \max L(\tilde\gamma)+\max H(\chi)$.
  \begin{itemize}
  \item For $G^{-1}(\theta_1)\le \eta_{p+m_2-1}\le \eta_{p+m_2}$, we check that $J(\mu,\theta_1,\eta_{p+m_2})$ equals
  \begin{equation}\label{tq}J(\mu,\theta_1,\eta_{p+m_2-1})+\sum\alpha_i \ln\frac{|\eta_i-\eta_{p+m_2-1}|}{|\eta_i-\eta_{p+m_2}|}+\theta_1(\eta_{p+m_2}-\eta_{p+m_2-1})\end{equation}
  so that by \eqref{maxH} we find
  $$\max I=J(\mu,\theta_1,\eta_{p+m_2})+J(\mu,\theta_2,\eta_{p+m_2-1})\,.$$
  \item  For $G_\mu^{-1}(\theta_1)\in [\eta_{p+m_2-1},\eta_{m+p_2}]$, $ J(\mu,\theta_2, G_\mu^{-1}(\theta_1))=J(\mu,\theta_2, G_\mu^{-1}(\theta_2))=J(\mu,\theta_2, \eta_{p+m_2-1}) $ since $\theta_1>\theta_2$ and $\eta_{p+m_2-1}
  <G_\mu^{-1}(\theta_1)<G_\mu^{-1}(\theta_2)$. Moreover as $\theta_1>G_\mu(\eta_{p+m_2})$,
  $$\max L= J(\mu,\theta_1,\eta_{p+m_2})+J(\mu,\theta_1,\eta_{m_2+p-1})+\ln \theta_1+1$$
  which again does not depend on $\eta_{m_2+p-1}$. Hence
  \begin{eqnarray*}
  \max I&=& J(\mu,\theta_2, \eta_{p+m_2-1})+\sum \alpha_i\ln|\eta_i-G_\mu^{-1}(\theta_1)|-\theta_1 G^{-1}_\mu(\theta_1)\\
  &&+ J(\mu,\theta_1,\eta_{p+m_2})
  +J(\mu,\theta_1, G^{-1}(\theta_1))+\ln \theta_1+1\\
  &=&  J(\mu,\theta_2, \eta_{p+m_2-1})+ J(\mu,\theta_1,\eta_{p+m_2})\\
  \end{eqnarray*}

  \item  For $G_\mu^{-1}(\theta_1)>\eta_{m+p_2}$, we compute
  \begin{eqnarray*}
  \max I&=& J(\mu,\theta_2, \eta_{p+m_2})+J(\mu,\theta_1, \eta_{p+m_2-1})\\
  &=&J(\mu,\theta_2, \eta_{p+m_2-1})+J(\mu,\theta_1, \eta_{p+m_2})\end{eqnarray*}
  since $\theta_2<\theta_1<G_\mu(\eta_{p+m_2})<G(\eta_{p+m_2-1})$ so that the above supremum does not depend on the outliers.

  \end{itemize}

In the case where  $N_{p+m_{2}}\ge 2$, we have $\chi =\eta_{p+m_{2}}$ by Weyl interlacing property and therefore it does not depend on the choice of $\gamma$. The result follows immediately after conditioning as in the proof above.
\end{proof}

A similar (but easier) argument  shows that
\begin{prop}\label{SI2dn}
Let $\theta_1\ge 0\ge  \theta_2$. Then, 
$$ \lim_{N\rightarrow\infty} \frac{1}{N} \log \E\left[ \exp \left( \frac{\beta N }{2}(\theta_1 \langle e , \bX_N e\rangle + \theta_2 \langle f, \bX_N f \rangle ) \right) \right]\qquad\qquad$$
$$\qquad\qquad=\frac{\beta}{2}(J(\mu,\theta_1,\eta_{p+m_2})+J(\mu,\theta_2,\eta_{1-m_{1}}))\,.
$$
Here $J$ is extended to negative values of $\theta_2$ by putting
$$J(\mu, \theta_2, \eta_{1-m_1})= K\left(\mu,\theta_2, \eta_{1-m_1}, v(\mu,\theta_2, \eta_{1-m_1})\right)
$$
with $K$ as in Lemma \ref{formulaI} and 
for $\theta\le 0,\eta\le l_\mu$, 
$$v(\mu,\theta, \lambda)=\left\lbrace\begin{array}{c l}
\lambda &\mbox{ if }  G_\mu(\lambda)\ge \theta,\cr
G_\mu^{-1}(\theta) &\mbox{ if }  G_\mu(\lambda)< \theta.\cr\end{array}\right.$$

\end{prop}
The formula for $J$ for negative $\theta$ can simply be found by replacing $\bX_N$ by $-\bX_N$.

\subsection{Limiting k-d spherical integrals}We now consider  more general $k$-dimensional spherical integrals with $k\ge 2$.
In the sequel we let $\lambda_{1}\ge\lambda_{2}\ge\cdots \ge\lambda_{k}\ge 0$ be  the limit of the $k$ largest outliers  of $\bX_{N}$ counted with multiplicity one, and  $\lambda_{-1}\le\cdots\le \lambda_{-\ell}\le0$  be the limit of the $\ell$ smallest outliers of $\bX_{N}$ counted with multiplicity one. With the previous notations, $\lambda_i=\eta_{p+m_2}$ for $i\in [1,N_{p+m_2}],$ $\lambda_i=\eta_{p+m_2-1}$ for $i\in [N_{p+m_2}+1,N_{p+m_2}+N_{p+m_2-1}]$.

\begin{prop}\label{SIkd} Fix two  integer numbers $k$ and $\ell$.
Let $(e_{1},\ldots e_{k}, e_{-1},\ldots e_{-\ell})$ be $k+\ell$ orthonormal vectors following the uniform law in the sphere and assume that the sequence $\bX_N$ has the form described at the beginning of this section. 
Let $\theta_1\ge \theta_2\ge\cdots \ge \theta_{k} \ge 0\ge\theta_{-\ell}\ge\cdots\ge \theta_{-1}$. Then
$$ \lim_{N\rightarrow\infty} \frac{1}{N} \log \E\left[ \exp \left( \frac{\beta N }{2}\sum_{i=-\ell, i \neq 0}^{k}\theta_{i}\langle e_{i},\bX_{N} e_{i}\rangle
 \right) \right]\qquad\qquad$$
$$\qquad\qquad=\frac{\beta}{2}\left(\sum_{i=1}^{k}
J(\mu,\theta_i,\lambda_{i})+\sum_{i=1}^{\ell}
J(\mu,\theta_{-i},\lambda_{-i})\right)\,.
$$

\end{prop}

\begin{proof}
For the sake of simplicity we will assume the outliers $\lambda_1,...,\lambda_{k}$ and $\lambda_{-1},...,\lambda_{-\ell}$ are distinct. The general case can be deduced by equicontinuity of the spherical integral.
We will prove this proposition by induction over $k+\ell$. We know it is true for $k+\ell\le 2$ by the previous section. By symmetry we can assume the proposition  true for $(\ell, k-1)$ and it is enough to show it still holds for $(\ell,k)$. 
Thus we will set $\eta_{p+m_2-i+1} = \lambda_i$ for $i \in [1,k]$, and $\eta_{- m_1 +i}= \lambda_{-i}$ and we will assume $N_{  p+m_2-i}=1$ for $i \in [1,k]$ and $N_{-m_1 + i} = 1$ for $i \in [1,\ell]$.
 We proceed as in the $2$-dimensional case by conditioning on the vector $e_1$ and so we have : 

\begin{eqnarray*}
 &\E\left[ \exp \left( \frac{\beta N }{2}\sum_{i=-\ell, i \neq 0}^{k}\theta_{i}\langle e_{i},\bX_{N} e_{i}\rangle
 \right) \right]   = \\
 &\E \left[ \exp \left( \frac{\beta N }{2}\theta_{1}\langle e_{1},\bX_{N} e_{1}\rangle
 \right)  \E\left[ \exp \left( \frac{\beta N }{2}\sum_{i=-\ell, i \neq 0 \atop i \neq 1}^{k}\theta_{i}\langle e_{i},\bX^{(e_1)}_{N} e_{i}\rangle
 \right)  \Big| e_1 \right] \right] 
\end{eqnarray*}

As previously, the eigenvalues of $\bX_N^{(e_1)}$ (seen as a $N-1 \times N-1$ matrix) are interlaced with those of $\bX_N$. Thus if we denote $\chi_j$ the $j$-th largest eigenvalue of $\bX_N^{(e_1)}$, $\chi_j$ is the unique solution in the interval $[\eta_{p+m_2 - j  }, \eta_{p+m_2 -j +1}]$ of the equation :
\begin{equation}\label{poiu}\sum_{i =1 - m_1}^{p+ m_2} \frac{\gamma_i(e)}{\chi_j - \eta_i} =0\end{equation}
for $j\in [1,k-1]$. The same equation holds for the $\ell$ smallest eigenvalues  below the bulk :   if we denote $\chi_{-j}$ the $j$-th smallest eigenvalue of $\bX_N^{(e)}$, it is solution of the same equation in $[ \eta_{-m_1 + j}, \eta_{ -m_1 +j +1}]$. Observe that unless $\gamma_i(e)$ vanishes, $\chi_i$ can not be equal to $\eta_i$. 
{
So, if we denote for $i = -m_1+ l+1,...,p+m_2-k$,   $\delta_i$ the solution of the same interlacing equation in $[\eta_i, \eta_i+1]$, up to diagonalization, $\bX_N^{(e_1)}$ has the following form :

\[ \bX_N^{(e_1)} = diag( \chi _{-1},...,\chi_{-l}, \underbrace{\eta_{-m_1+l+1}}_{N_{ -m_1+l+1} -1 }, \delta_{-m_1+l+1}, \underbrace{\eta_{-m_1+l+2}}_{N_{ -m_1+l+2} - 1},...,\underbrace{\eta_{p+ m_2 - k}}_{N_{ p+ m_2 - k} -1 }, \delta_{p+ m_2 -k}, \chi_{k-1},...,\chi_1 ) \]
where the $\delta_j$ and the $\chi_i$ being continuous functions of $\gamma(e)$.} 
If we denote $\chi_i(\kappa)$ the value of $\chi_i$ when $\gamma(e) = \kappa$, we deduce  by induction  and using the continuity in Remark \ref{SICont} that when $\gamma({e_1})$ converges toward $\kappa$ : 

\[ \lim_{N \to \infty} \frac{1}{N-1} \ln \E \left[ \exp \left( \frac{\beta N }{2}\sum_{i=-\ell, i \neq 0, 1}^{k}\theta_{i}\langle e_{i},\bX^{(e_1)}_{N} e_{i}\rangle\right)  \Big| e_1 \right] = \frac{\beta}{2}\sum_{i=-\ell, i \neq 0, 1}^{k} J(\mu,\theta,\chi_i(\kappa)) \]

Then, using again that $\gamma^{N}$ satisfies a large deviations principle we can write that : 

\begin{eqnarray}
&&\lim_{N \to \infty} \frac{1}{N} \ln \E\left[ \exp \left( \frac{\beta N }{2}\sum_{i=-\ell, i \neq 0}^{k}\theta_{i}\langle e_{i},\bX_{N} e_{i}\rangle
 \right) \right] \label{bn}\\
&&=\frac{\beta}{2}\sup_{\gamma \in (\R^{+})^{p+m_1+m_2}, \sum \gamma_i =1} \bigg\{ \sum_{i=-\ell}^{-1}  J(\mu,\theta_i,\chi_i(\gamma)) + \sum_{i=2}^{k} J(\mu,\theta_{i},\chi_{i-1}(\gamma)) \nonumber\\
&&\qquad\qquad+ \sum_{i=1}^p \alpha_i \log \frac{\gamma_i}{\alpha_i} + \sum_{i = 1- m_1}^{p+m_2} \theta_1 \eta_{i} \gamma_i \bigg\} \nonumber
\end{eqnarray}

By continuity, taking this supremum only on the set of $\gamma$  summing up to  $1$ and such that $\gamma_i >0$ does not change its value. Notice that for such $\gamma$ we have for all $i$ and $j$ $\chi_i(\gamma) \neq \eta_j$. We set 
$I_{-j} = ] \eta_{-m_1+j}, \eta_{-m_1 + j +1}[$ for $j =1,...,\ell$ and $I_j =] \eta_{m_2 + p -j }, \eta_{m_2 + p -j+1 }[$ for $j =1,...,k-1$ and we define :
\[ D = \Big\{ (\gamma,\chi) \in \prod_{j=- \ell, j \neq 0}^{k-1} I_j \times (\R^{+,*})^{m_1+m_2+p} : \sum_i \gamma_i = 1,  \sum_i \frac{\gamma_i}{\chi_j - \eta_i} = 0\, \forall j \in [-\ell, k-1] \setminus \{0 \}\Big\} \]
Therefore we have :
\begin{eqnarray}
&&\lim_{N \to \infty} \frac{1}{N} \ln \E\left[ \exp \left( \frac{\beta N }{2}\sum_{i=-\ell, i \neq 0}^{k}\theta_{i}\langle e_{i},\bX_{N} e_{i}\rangle
\right) \right] \\
&&=\frac{\beta}{2}\sup_{(\chi,\gamma) \in D} \bigg\{ \sum_{i=-\ell}^{-1}  J(\mu,\theta_i,\chi_i) + \sum_{i=2}^{k} J(\mu,\theta_{i},\chi_{i-1}) 
+ \sum_{i=1}^p \alpha_i \log \frac{\gamma_i}{\alpha_i} + \sum_{i = 1- m_1}^{p+m_2} \theta_1 \eta_{i} \gamma_i \bigg\} \nonumber
\end{eqnarray}

For $i = -m_1 +\ell,...,m_2 +p -k$, we set :
$$\bar{\gamma}_i = \frac{\prod_{j=1}^{\ell} (\eta_{-m_1+j} - \eta_{i})\prod_{j=1}^{k -1}(\eta_{m_2+p- j +1} - \eta_{i})  }{\prod_{j=1}^{\ell} (\chi_{-j} - \eta_{i})\prod_{j=1}^{k -1}(\chi_{j} - \eta_{i}) }\gamma_i$$
We have that if $\gamma_i >0$ for all $i$ then $\bar{\gamma}_i > 0$ for all $i$ and $\bar\gamma_i$ vanishes at the outliers. We want to prove that this definition provides a one to one correspondance between the set $D$ of parameters $(\chi,\gamma)$ and the set $\bar{D}$ parameters $(\chi,\bar{\gamma})$ defined as follows : 
\[ \bar{D} = \Big\{ (\chi,\bar{\gamma}) \in \prod_{j=- \ell, j \neq 0}^{k-1} I_j \times (\R^{+,*})^{m_1+m_2+p-k-\ell+1} , \sum_i \bar{\gamma}_i = 1\Big\} \]
Note that $\bar\gamma$ lives a priori in a set of  $k+\ell-1$ dimension  smaller but $\gamma$ was satisfying  as well $k+\ell-1$ additional equations. 
First, let us prove that if $(\chi,\gamma) \in D$ the $\bar{\gamma}_i's$ sum up to $1$. We let for a  real number $X$, $F$ to be the rational function 
$$F(X) = \frac{\prod_{j=1}^{\ell} (\eta_{-m_1+j} - X)\prod_{j=1}^{k -1}(\eta_{m_2+p- j +1} - X)  }{\prod_{j=1}^{\ell} (\chi_{-j} - X)\prod_{j=1}^{k -1}(\chi_{j} - X) }$$
 so that $\bar{\gamma}_i = F(\eta_i) \gamma_i$ for $i \in [-m_1 +\ell,m_2 +p -k]$  and $F(\eta_i) =0$ for the other values of $i$. Let us decompose $F$ in partial fractions : as it goes to one at infinity, we find 

\[ F(X) = 1 + \sum_{j=- \ell, j \neq 0}^{k-1} \frac{a_j}{ \chi_{j} - X} \]
for some real numbers $a_j$.
Then, since $F(\eta_i) =0$ for $i\neq -m_1 +\ell,...,m_2 +p -k$ we have :
\begin{eqnarray*} \sum_{i = -m_1 +\ell}^{m_2+p-k} \bar{\gamma}_i  &= &\sum_{i = -m_1 +1}^{m_2+p} F(\eta_i)\gamma_i \\
	&= &\sum_{i = -m_1 +1}^{m_2+p} \gamma_i  + \sum_{j=- \ell,j \neq0}^{k-1}a_j \sum_{i = -m_1 +1}^{m_2+p} \frac{\gamma_i}{\chi_j - \eta_i} \\
	&= & 1 
\end{eqnarray*}
where we used the interlacing relations. Therefore, since when $\chi$ is fixed the fonction $\gamma \mapsto \bar{\gamma}$ is an affine map between the affine subspace $E$ of $\R^{ p+ m_1+m_2}$ defined by  the $k+\ell-1$ interlacing relations and the condition of sum one and the affine subspace $F$  of $\R^{ p+ m_1+m_2 - k -\ell +1}$ defined by the condition of sum one. Since these spaces have the same dimension to conclude we only need to prove that this map is injective and that for all $\gamma \in E$, $\bar{\gamma}_i >0$ for all $i$ implies $\gamma_i >0$ for all $i$. To prove injectivity first notice that $\gamma_i=F(\eta_i)^{-1}\bar\gamma_i$ for $i \in [-m_1 +\ell,m_2 +p -k]$. We next show how to reconstruct $\gamma_i$  for $i \in [-m_1 +\ell,m_2 +p -k]^c$, and more precisely $i\in [m_2 +p -k+1, m_2+p]$. 
To this end,
for $j=1,...,k-1$, we let $G_j(X) = \frac{F(X)}{(\eta_{p+m_2 -j+1} -X)}$ and for $j=1,...,\ell$, we let $G_{-j}(X) = \frac{F(X)}{(\eta_{-m_1+j} -X)}$. Let us suppose $j >0$ (the $j <0$ case is similar). Then again decomposing $G_j$ in partial fractions, we have 
\[  G_j(X) = \sum_{j=- \ell, j' \neq 0,j}^{ k-1} \frac{b_{j'}}{\chi_{j'} -X} \]
for some real numbers $b_j$. Again by the interlacing relations
\[ \sum_{i= - m_1 +1}^{p + m_2}  G_j(\eta_i)\gamma_i = 0\,. \]
But we can also write : 
\[ \sum_{i= - m_1 +1}^{p + m_2}  G_j(\eta_i)\gamma_i = \sum_{i= - m_1 + \ell}^{p + m_2 - k } \frac{ \bar{\gamma_i}}{\chi_j - \eta_i} + G_j(\eta_{p+m_2-j+1}) \gamma_{p+m_2-j+1} \]
so that we deduce
\[ \gamma_{p+m_2-j+1} = - (\sum_{i= - m_1 = \ell}^{p + m_2 - k +1} \frac{ \bar{\gamma_i}}{\chi_j - \eta_i}) / G_j(\eta_{p+m_2-j+1})\,. \]
As a consequence, the map $\gamma \mapsto \bar{\gamma}$ is injective. 
Furthermore if $\bar{\gamma}_i >0$ for all $i$, then $\gamma_j > 0$ since $G_j(\eta_j) < 0$. The same remains true for $j$ negative.
Therefore we have that the change of variables from $(\chi, \gamma)\in D$ to $(\chi,\bar{\gamma}) \in \bar{D}$ is one to one. But before changing variables, let us compare $\sum \eta_i \gamma_i$ and $\sum \eta_i \bar{\gamma_i}$. We use the following decomposition : 
\[ XF(X) = X + S + \sum_{j=- \ell,j \neq 0}^{k-1} \frac{c_j}{\chi_j - X} \]
for some real numbers $c_j$ and where 
\[ S = \sum_{j= - \ell,j \neq 0}^{k-1} \chi_j - \sum_{j=1}^{\ell} \eta_{- m_1+j}  - \sum_{j=1}^{k-1} \eta_{m_2 +p -j +1} \,. \]
We deduce that
\[ \sum_{ - m_1 + l}^{m_2 +p -k +1} \eta_i \bar{\gamma}_i = \sum_{ - m_1 +1}^{m_2 +p} \eta_i F(\eta_i) \gamma_i = \sum_{ - m_1 +1}^{m_2 +p} \eta_i  \gamma_i + S \]
where we used again the interlacing relationships and the fact that the $\gamma_i$'s sum up to $1$. Coming back to \eqref{bn}, 
we have to take the supremum of the following function $I$ for $( \chi,\bar{\gamma})\in \bar{D}$: 

\begin{eqnarray*}
 I(\bar{\gamma}, \chi) &=& \sum_{j=1}^{k-1} \left[ J(\mu,\theta_{j+1},\chi_j) + \sum_{i=1}^{p} \alpha_i \log  |\chi_j - \eta_{i}| - \sum_{i=1}^{p} \alpha_i \log  |\eta_{m_2+p-j+1} - \eta_{i}| \right] \\
& +  &
\sum_{j=-\ell}^{-1} \left[ J(\mu,\theta_{j},\chi_j) + \sum_{i=1}^{p} \alpha_i \log  |\chi_j - \eta_{i}| -\sum_{i=1}^{p} \alpha_i \log  |\eta_{-m_1 - j} - \eta_{i}| \right] \\
&+& \sum_{i=1}^p \alpha_i \ln  \frac{\bar{\gamma_i}}{\alpha_i} + \theta_1\left( \sum_{i= -m_1 +1}^{p+m_2} \eta_i \bar{\gamma}_i -  \sum_{j=- \ell,j \neq 0}^{k-1} \chi_j +  \sum_{j=1}^\ell \eta_{- m_1+j}  +  \sum_{j=1}^{k-1} \eta_{m_2 +p -j +1}\right)
\end{eqnarray*}

Therefore we have :
\begin{eqnarray*} I(\bar{\gamma}, \chi) &=& \sum_{i=1}^{k-1} H(\chi_{i},\theta_{i+1} ) + \sum_{i=-\ell}^{-1} H(\chi_{i},\theta_{i} ) + \sum_{i=1}^p \alpha_i \log \frac{\bar{\gamma}_i}{\alpha_i} + \theta_1 \sum_{i=m_1 -\ell}^{p+m_2-k +1} \eta_i \bar{\gamma}_i \\ &-& \sum_{j=1}^{k-1} \sum_{i=1}^{p} \alpha_i \log  |\eta_{m_2+p-j+1} - \eta_{i}| - \sum_{j=-\ell}^{-1}\sum_{i=1}^{p}\alpha_i \log  |\eta_{-m_1 - j} - \eta_{i}| \\ &\jonathan{+}& \theta_1 \sum_{j=1}^{\ell} \eta_{- m_1+j}  \jonathan{+} \theta_1 \sum_{j=1}^{k-1} \eta_{m_2 +p -j +1}
\end{eqnarray*}
where  we set : 

\[ H(\chi,\theta) = J(\mu,\theta,\chi) + \sum_{i=1}^p \alpha_i \ln |\chi - \eta_i| -  \chi \theta_1\,. \]
The supremum over $\bar\gamma$ and $\chi$ are now decoupled and the $\chi_i$ belongs to $]\eta_{-m_1+i},\eta_{-m_1 +i +1}[$ if $i\in [-\ell,-1]$ and $]\eta_{p+m_2-i+1},\eta_{p+m_2-i+2}[$ if $i\in [1,k]$. 
As in the two-dimensional case we can compute  for $i =2,...,k$,

$$\sup_{\chi\in ]\eta_{p+m_2-i+1},\eta_{p+m_2-i+2}[}
H(\chi,\theta_i)=\left\lbrace
\begin{array}{ll}
H(\eta_{p+m_2-i+1}, \theta_{i} )& \mbox{ if }  G_\mu^{-1}(\theta_1)\le \eta_{p+m_2-i+1},\cr
H( G_{\mu}^{-1}(\theta_{1}), \theta_{i} )&\mbox{ if }  G_\mu^{-1}(\theta_1)\in [\eta_{p+m_2-i+1},\eta_{m+p_2-i+2}],\cr
H(\eta_{p+m_2-i+2}, \theta_{i}) &\mbox{ if }  G_\mu^{-1}(\theta_1)>  \eta_{p+m_2-i+2}.\cr\end{array}\right. $$
Moreover, for $i=1, ...,\ell$, $H(\chi,\theta_{-i})$ is a decreasing function of $\chi$ since $\theta_{-i}$ is negative and so 

\[ \sup_{\chi\in ]\eta_{-m_1+i},\eta_{-m_1 +i +1}[} H(\chi,\theta_{-i})=H( \eta_{-m_1+i},\theta_{-i})\,. \]
It remains to optimize the sum of the third and fourth term in $I(\bar{\gamma},\chi)$. But this sum is equal to $I_{\theta_1,\eta_{p+m_2 - k+1}}^{p+m_2 -k+1}(\bar{\gamma})$, see \eqref{defI}. Thus, taking the supremum for $\bar{\gamma}_i > 0$ and $\sum \bar{\gamma}_i =1$ gives  $J(\mu,\theta_1,\eta_{p+m_2-k+1})$.

To conclude, we need to look at the position of $G^{-1}_{\mu}( \theta_1)$ relatively to the $k$ largest outliers. 
Let us denote $H_i = \max H (.,\theta_i)$ for $i < 0$ and $H_i = \max H(.,\theta_{i+1})$ for $i >0$. We have

\begin{eqnarray*}
 \sup I &=& \sum_{j=1}^{k-1} \left( H_j - \sum \alpha_i \log | \eta_i - \eta_{p+m_2 - j +1}| + \theta_1 \eta_{p+m_2 - j +1} \right)\\
&+& \sum_{j=1}^{\ell} \left( H_{-j} - \sum \alpha_i \log | \eta_i - \eta_{-m_1+j}| + \theta_1 \eta_{-m_1 +j}  \right) \\
&+& J(\mu,\theta_1, \eta_{p+m_2-k+1})
\end{eqnarray*}
Here we will treat the case where $G^{-1}_{\mu}( \theta_1) \in [\eta_{p+m_2 -k +1}, \eta_{p+m_2}]$ which is the most complex one.
First of all, since for $j=1,...,\ell$ $H_{-j} = H( \eta_{-m_1+ j},\theta_{-j})$, we have that in the second sum, the term of index $j$ is indeed equal to $J(\mu,\theta_{-j},\eta_{-m_1+j})$.  If $j'$ is the index such that, $G^{-1}_{\mu}( \theta_1) \in [\eta_{p+m_2-j'}, \eta_{p+ m_2 -j'+1}]$ then for $j <   j'$, $H_{j}= H(\eta_{ p + m_2 -j},\theta_{j+1})$ 
and the term of index $j$ of the first sum is :

\[ J(\mu,\theta_{j+1}, \eta_{ p + m_2 - j }) + \sum \alpha_i \log \frac{| \eta_{p+m_2 -j} - \eta_i |}{|\eta_{p+m_2 -j +1} - \eta_i|} + \theta_1 ( \eta_{p+m_2 - j } - \eta_{p+m_2 - j +1}) \]

The term of index $j'$ is equal to : 
\[ J(\mu,\theta_{j+1},G^{-1}_\mu(\theta_1)) + \sum_i \alpha_i \log \frac{| G^{-1}_\mu(\theta_1) - \eta_i|}{|\eta_{p + m_1 -j' +1} - \eta_i|} + \theta_1 ( \eta_{p+m_2 - j'} - G^{-1}_\mu(\theta_1)) \]

And the terms $j > j'$ are equal to $J( \mu, \theta_{j+1},\eta_{p+m_2 - j +1})$.  Since $\theta_{j+1} \leq \theta_1$, $G_{\mu}^{-1}(\theta_{j+1}) \geq G_{\mu}^{-1}(\theta_{1})$, so we have that for $j >j'$
\[ J( \mu, \theta_{j+1},\eta_{p+m_2 - j +1}) = J( \mu, \theta_{j+1},\eta_{p+m_2 - j }) \mbox{ and }
 J(\mu,\theta_{j'+1},G^{-1}_\mu(\theta_1)) = J( \mu, \theta_{j'+1},\eta_{p+m_2 - j' }) \]
Therefore the whole sum can be simplified as follows : 

\begin{eqnarray*}
 \max I &=& \sum_{j=1}^{k-1} J(\mu,\theta_{j+1},\eta_{p+m_2-j})+ \sum_{j=1}^{\ell} J(\mu,\theta_{-j},\eta_{-m_1+j}) +\sum_i \alpha_i \ln \frac{|G^{-1}_{\mu}(\theta_1) - \eta_i|}{|\eta_{p+m_2} - \eta_i|} \\&+ &\theta_1(\eta_{p+m_2} - G^{-1}_{\mu}(\theta_1) ) 
+ J(\mu,\theta_1, \eta_{p+m_2-k+1})
\end{eqnarray*}
Then we  notice that $ J(\mu,\theta_1, \eta_{p+m_2-k+1}) = J(\mu,\theta_1,G^{-1}_{\mu}(\theta_1) )$ and  conclude since : 
\[ J(\mu,\theta_1, \eta_{p+m_2}) = J(\mu,\theta_1,G^{-1}_{\mu}(\theta_1) ) + \sum_i \alpha_i \ln \frac{|G^{-1}_{\mu}(\theta_1) - \eta_i|}{|\eta_{p+m_2} - \eta_i|} + \theta_1(\eta_{p+m_2} - G^{-1}_{\mu}(\theta_1) )\,. \]



\end{proof}
\section{Diffuse spectrum}\label{dif}
We next consider the general case where $\bX_{N}$ is a Hermitian matrix such that
$$\hat\mu_{\bX_{N}}=\frac{1}{N}\sum_{i=1}^{N}\delta_{\lambda_{i}}$$
converges towards a probability measure $\mu$ with support with rightmost point $r_{\mu}$ and leftmost point $l_{\mu}$ which are assumed to be finite.   Let  $\lambda_{1}^{N}\ge \lambda_{2}^N\ge\cdots \ge \lambda_{k}^{N}\ge r_{\mu}$ be  the  $k$ largest outliers  of $\bX_{N}$ counted with multiplicity one, $\lambda_{N }^{N}\le\cdots\le \lambda_{N-\ell +1}^{N}\le 0$  be the smallest outliers of $\bX_{N}$ with multiplicity one (but eventually equal).
Assume that
$$\lim_{N\rightarrow\infty}\lambda_{i}^{N}=\lambda_{i}>r_{\mu} \mbox{ for } i\in [1,k], \lim_{N\rightarrow\infty}\lambda_{N-i +1}^{N}=\lambda_{-i}<l_{\mu} \mbox{ for } i\in [1,\ell]\,.$$
If the above assumption is true for the $k$ largest outliers and we want to study the spherical integral with non-negative $\theta_i$'s while  relaxing the assumption on the smallest ones, we still need to assume that the latter are bounded.

We are going to  prove that
\begin{prop}\label{SIkdc} Fix two integer numbers $k,\ell$.
Let $\theta_1\ge \theta_2\ge\cdots \ge \theta_{k} \ge 0\ge\theta_{- \ell}\ge\cdots\ge \theta_{-1}$. Then
$$ \lim_{N\rightarrow\infty} \frac{1}{N} \log \E\left[ \exp \left( \frac{\beta N }{2}\sum_{i=- \ell, i \neq 0}^{k}\theta_{i}\langle e_{i},\bX_{N} e_{i}\rangle
 \right) \right]\qquad\qquad$$
$$\qquad\qquad=\frac{\beta}{2}\left(\sum_{i=1}^{k}
J(\mu,\theta_i,\lambda_{i})+\sum_{i=1}^{\ell}
J(\mu,\theta_{-i},\lambda_{-i})\right)\,.
$$

\end{prop}
\begin{proof}
We first remark that we can assume  $\bX_{N}$ diagonal without loss of generality. In a first step,  we  assume that the the partition function of $\mu$ is continuous and that  $\bX_N$ has bounded norm.
We fix $\varepsilon>0$ . We know by assumption that for $N$ large enough, the spectrum of $\bX_N$ is included into $[\lambda_{-\ell}-\varepsilon, \lambda_k+\varepsilon[$. 
For $j\ge -1$ we let $n_j^{\varepsilon,N}$ be the number of eigenvalues of $\bX_N$ in $[\lambda_{-\ell}+j\varepsilon, \lambda_{-\ell}+(j+1)\varepsilon[$, until the first $j$ so that $\lambda_{-\ell}+(j+1)\varepsilon\ge \lambda_k+\varepsilon$. We let $\bX_N^\varepsilon$ be the diagonal matrix with eigenvalues $\lambda_{-\ell}+j\varepsilon$ with multiplicity $n_j^{\varepsilon,N}$.  Since we assumed the extreme eigenvalues distinct, for $N$ large enough the extreme eigenvalues of $\bX_N^\varepsilon$ are $(\lambda_i, i\le k, \lambda_{- j}, j\le \ell)$.
$n_j^{\varepsilon,N}/N$ converges towards $\mu([\lambda_{-\ell}+j\varepsilon, \lambda_{-\ell}+(j+1)\varepsilon])$ since we assumed the partition function of $\mu$ to be continuous.  Hence, $\bX_N^\varepsilon$ satisfies the hypotheses of Proposition \ref{SIkd}. Moreover, by definition, for $N$ large enough we know that
$$\|\bX_N-\bX_N^\varepsilon\|\le \varepsilon$$
Therefore,
$$\left|\frac{1}{N} \log \frac{\E\left[ \exp \left( \frac{\beta N }{2}\sum_{i=- \ell ,i \neq 0}^{k}\theta_{i}\langle e_{i},\bX_{N} e_{i}\rangle
 \right) \right]}{  \E\left[ \exp \left( \frac{\beta N }{2}\sum_{i=- \ell, i \neq 0}^{k}\theta_{i}\langle e_{i},\bX_{N}^\varepsilon e_{i}\rangle
 \right) \right]}\right|\le \frac{\beta}{2}\sum|\theta_i|\varepsilon\,.$$
 On the other hand, Proposition \ref{SIkd} implies
 $$\lim_{N\rightarrow\infty} \frac{1}{N} \log \E\left[ \exp \left( \frac{\beta N }{2}\sum_{i=- \ell, i \neq 0}^{k}\theta_{i}\langle e_{i},\bX_{N}^\varepsilon e_{i}\rangle
 \right) \right]=\frac{\beta}{2}\left(\sum_{i=- \ell , i \neq 0}^{k}
J(\mu^\varepsilon,\theta_i,\lambda_{i})\right)\,
$$
with $\mu^\varepsilon=\sum_j \mu([\lambda_{-\ell}+j\varepsilon, \lambda_{-\ell}+(j+1)\varepsilon])\delta_{\lambda_{-\ell}+j\varepsilon}$. By continuity of $\mu\rightarrow J(\mu,\theta_i,\lambda_{i})$, see \cite{Ma07} or the Appendix, and the weak convergence of $\mu^\varepsilon$ towards $\mu$, the conclusion follows.

   Finally,  to remove the condition that $\mu$ has a continuous partition function we note that we can always add a small  matrix to $\bX_{N}$ and its contribution will go to zero as its norm goes to zero after $N$ goes to infinity. We again assume $\bX_N$  diagonal  and replace it by the diagonal matrix with the same 
   outliers and in the bulk the entries are added independent uniform variables with uniform distribution on $[0,\varepsilon]$. Again $\bX_N^\varepsilon-\bX_N$ has norm bounded by $\varepsilon$. Moreover, the spectral measure of $\bX_N^\varepsilon$ converges towards
   $  \mu*1_{[0,\varepsilon]}du/\varepsilon$ whose  partition function is continuous.  Hence, we can apply our result to
   this new matrix and then let $\varepsilon$ go to zero to conclude.
\end{proof}
\section{Applications to large deviations for the extreme eigenvalues of random matrices}\label{appl}

\subsection{Universality of the large deviations for the $k$ extreme eigenvalues of Wigner matrices with sharp sub-Gaussian entries}
In this  section, we prove Theorem \ref{thmuniv}.
The  proof follows the ideas of \cite{HuGu} quite closely: we simply sketch the main arguments and changes. First note that it is enough to prove a weak large deviations principle thanks to our assumption which insures that exponential tightness holds. Moreover. let $\bar \lambda^N=(\lambda_1,\ldots,\lambda_k,\lambda_{N-k},\ldots,\lambda_N)$ be the $k$ extreme eigenvalues of $\bX_N$.  To get a weak large deviations upper bound, we tilt the measure by spherical integrals as above : if $(e_i)_{-k\le i\le k}$ follows the uniform law on the set of $2k$ orthonormal vectors on the sphere ($\theta_0=0$ and $e_0=0$ is added to shorten the notations), $\theta_i $ are real numbers of the same sign than  $i\in [-k,k]$ to be chosen later, we write 
\begin{eqnarray*}
\mathbb P\left( \|\bar\lambda^N-\bar x\|_2\le \varepsilon\right)&\le& \E_{\bX_N}\left[
1_{\|\bar\lambda^N-\bar x\|_2\le \varepsilon}\frac{\E_e\left[ \exp \left( \frac{\beta N }{2}\sum_{i=-k}^{k}\theta_{i}\langle e_{i},\bX_{N} e_{i}\rangle
 \right) \right]}{\E_e\left[ \exp \left( \frac{\beta N }{2}\sum_{i=-k}^{k}\theta_{i}\langle e_{i},\bX_{N} e_{i}\rangle
 \right) \right]}\right]\\
 &\le& e^{-N\frac{\beta}{2} F(\bar x,\bar \theta)+o(\varepsilon)N }\E_{\bX_N}
\E_e\left[ \exp \left( \frac{\beta N }{2}\sum_{i=-k}^{k}\theta_{i}\langle e_{i},\bX_{N} e_{i}\rangle
 \right) \right]\\
 \end{eqnarray*}
 where  
 $$F(\bar x,\bar\theta)= \sum_{i=-k}^{k}
J(\sigma,\theta_i,x_{i})\,.$$ 
 We used in the second line  that by Theorem \ref{unifconv}, the spherical integrals are uniformly continuous and are asymptotically given by $F(\bar x,\bar\theta)$, and 
  our assumption that the spectral measure of $\bX_N$ converges towards the semi-circle law $\sigma$ faster than any exponential. Here  $o(\varepsilon)$ goes to zero when $\varepsilon$ does. We also used the bound
 \begin{equation}\label{ineq}
\frac{ \E_{\bX_N}\left[
1_{\|\bar\lambda^N-\bar x\|_2\le \varepsilon}\E_e\left[ \exp \left( \frac{\beta N }{2}\sum_{i=-k}^{k}\theta_{i}\langle e_{i},\bX_{N} e_{i}\rangle
 \right) \right]\right]}{ \E_{\bX_N}\left[\E_e\left[ \exp \left( \frac{\beta N }{2}\sum_{i=-k}^{k}\theta_{i}\langle e_{i},\bX_{N} e_{i}\rangle
 \right) \right]\right]}\le 1
 \end{equation}
 We next compute the expectation of the spherical integral by using that our entries are sharp-subgaussian :
\begin{eqnarray*}
&&\E_{\bX_N}\left[
\E_e\left[ \exp \left( \frac{\beta N }{2}\sum_{i=-k}^{k}\theta_{i}\langle e_{i},\bX_{N} e_{i}\rangle
 \right) \right]\right]\\
 &&\le \E_e[ \exp\{ \frac{\beta}{4} N\sum_{i=-k}^{k}\sum_{k\le j} 2^{k\neq j} |\sum \theta_i e_i(k)e_i(j)|^2\}]= \exp\{\frac{\beta}{4} \sum_{j=-k}^k \theta_j^2\}\end{eqnarray*}
 We hence get the upper bound
 $$\limsup_{\varepsilon\ra 0}\limsup_{N\ra\infty} \frac{1}{N}\ln \mathbb P\left( \|\bar\lambda^N-\bar x\|_2\le \varepsilon\right)\le -\frac{\beta}{2}\sup_{\theta_i}\{ \sum_{j=-k}^k \frac{\theta_j^2}{2}-F(\bar x,\bar\theta)\}$$
 where we take the supremum over non-negative $\theta_i$ for $i\in [1, k]$ and non-positive $\theta_i$'s for $i\in [-k,-1]$. Finally we observe that the supremum decouples and recall from \cite[Section 4.1]{HuGu} that the supremum over each $\theta_i$ of $\theta_i^2/2-
 J(\sigma,\theta_i,x_{i})$ gives $\int_{2}^{|x_{i}|}\sqrt{t^{2}-4}dt$. 
 To get the lower bound, we need to show that there exists $\bar\theta=(\theta_{-k},\ldots,\theta_{-1},\theta_1,\ldots,\theta_k)$ such that \eqref{ineq} is almost an equality in the sense that for every $\varepsilon>0$
 \begin{equation}\label{bbp}\liminf_{\N\ra\infty}\frac{1}{N}\ln \frac{\E_{\bX_N}\left[
1_{\|\bar\lambda^N-\bar x\|_2\le \varepsilon}\E_e\left[ \exp \left( \frac{\beta N }{2}\sum_{i=-k}^{k}\theta_{i}\langle e_{i},\bX_{N} e_{i}\rangle
 \right) \right]\right]}{ \E_{\bX_N}\left[\E_e\left[ \exp \left( \frac{\beta N }{2}\sum_{i=-k}^{k}\theta_{i}\langle e_{i},\bX_{N} e_{i}\rangle
 \right) \right]\right]}\ge 0\end{equation}
 and
 \begin{equation}\label{bia}\liminf_{\N\ra\infty}\frac{1}{N}\ln\E_{\bX_N}\left[
\E_e\left[ \exp \left( \frac{\beta N }{2}\sum_{i=-k}^{k}\theta_{i}\langle e_{i},\bX_{N} e_{i}\rangle
 \right) \right]\right]\ge \frac{\beta}{2} \sum_{j=-k}^k \theta_j^2\,.\end{equation}
 In both cases we use the fact that under the uniform measure, the vectors $e_i$ are delocalized with overwhelming probability, namely if $V_N^\kappa=\cap_{1\le i\le k}\{\|e_i\|_\infty\le N^{-1/4-\kappa}\}$ then $\mathbb P(V_N^\kappa)$ goes to one for any $\kappa\in (0,1/4)$. Therefore, to prove \eqref{bia} we notice that
 \begin{eqnarray*}
&&\E_{\bX_N}\left[
\E_e\left[ \exp \left( \frac{\beta N }{2}\sum_{i=-k}^{k}\theta_{i}\langle e_{i},\bX_{N} e_{i}\rangle
 \right) \right]\right]\\
 &\ge& \E_e\left[1_{e\in V_N^\kappa}\prod_{i\le j} \E[\exp\{ \frac{\beta}{2}N 2^{1_{i\neq j}}\sum_{r} \theta_r \Re(e_r(i)\bar  e_r(j) X_{ij})]\right]\\
 &\ge& \exp\{N\frac{\beta}{2} \sum_{r=-k}^k \theta_j^2+O(N^{1-2\kappa})\}\mathbb P(V_N^\kappa)\end{eqnarray*}
 where we used that $\sum_{r} \theta_r e_r(i)\bar  e_r(j)$ is of order at most $N^{-1/2-2\kappa}$ on $V_N^\varepsilon$ so that we can expand the Laplace transform of the entries around the origin. This proves \eqref{bia}. To prove \eqref{bbp} we notice that it is enough to show that for $N$ large enough
 $$\inf_{\bar e\in V_N^\varepsilon}\frac{\E_{\bX_N}\left[
1_{\|\bar\lambda^N-\bar x\|_2\le \varepsilon}\exp \left( \frac{\beta N }{2}\sum_{i=-k}^{k}\theta_{i}\langle e_{i},\bX_{N} e_{i}\rangle
 \right) \right]}{ \E_{\bX_N}\left[ \exp \left( \frac{\beta N }{2}\sum_{i=-k}^{k}\theta_{i}\langle e_{i},\bX_{N}e_{i}\rangle
 \right) \right]}\ge \frac{1}{2}\,.$$
But under the law tilted by $\exp \left( \frac{\beta N }{2}\sum_{i=-k}^{k}\theta_{i}\langle e_{i},\bX_{N} e_{i}\rangle
 \right) $, $\bX_N$ still has independent entries. We can compute its mean and covariance under the tilted law and using again that $\sum_{r} \theta_r e_r(i)\bar  e_r(j)$ is of order at most $N^{-1/2-2\kappa}$, we see that its mean is $\sum \theta_i e_i e_i^*$ and its covariance is close to $1/N$. We deduce as in \cite{HuGu} and the BBP transition \cite{BBP,PiReSo} that under this tilted law the outliers of $\bX_N$ are given by $\theta_i+\theta_i^{-1}$: it is therefore sufficient to choose $\theta_i=\frac{1}{2}(x_i\pm\sqrt{x_i^2-4})$. We refer the reader to \cite{HuGu} for more details. 

\subsection{Universality of the large deviations for the $k$ largest eigenvalues of Wishart matrices with sharp sub-Gaussian entries} We here prove Theorem \ref{thmuniv3}
and, as in the previous subsection, we will only sketch the changes from the proof in \cite{HuGu}. As in \cite{HuGu} we will study the largest eigenvalues of the linearized matrix $\bY_N$ of the matrix ${N^{-1/2}}\bG_{L,M}\bG_{L,M}^*$ : 
\[\bY_N = \begin{pmatrix} 0_{L \times L} & \frac{1}{\sqrt{N} } \bG_{L,M} \\\frac{1}{\sqrt{N}}\bG_{L,M}^* & 0_{M \times M} \end{pmatrix}\]
Up to a factor $(N/L)^{1/2} = ( (1+ \alpha) + o(N^{-\kappa}))^{1/2}$ $\bY_N$ is the linearization of $\bW_{L,M}$. 
 The main difference with the proof for Wigner matrices will be that computing the  asymptotics  of the annealed spherical integral requires more skill as it depends on the large deviations for the scalar products of projections of vectors uniformly distributed on the sphere : we can not merely assume that they are delocalized since this could a priori change the large deviations weight. To be more precise, let $\Lambda_N$ be the annealed spherical integral given for $\bar\theta=(\theta_1,\ldots,\theta_k)\in(\mathbb R^+)^k
 $ by

\[ \Lambda_N (\bar{\theta}) = \frac{1}{N} \log \E_{\bX_N}\left[
\E_e\left[ \exp \left( \frac{\beta N }{2}\sum_{i=1}^{k}\theta_{i}\langle e_{i},\bY_{N} e_{i}\rangle
\right) \right]\right]\,. \]

We shall prove that
\begin{equation}\label{lim}\lim_{N\rightarrow \infty}\frac{1}{N}\log \Lambda_N(\bar\theta)=\Lambda(\bar\theta)=\sum_{i=1}^k \Lambda(\theta_i)\end{equation}
with, if $\alpha'=(1+\alpha)^{-1}$, and with $\alpha$ the limit of $M/N$,
$$\Lambda(\theta)=  \sup_{  a\in ]0,1[}   \Big( \theta^2 a(1- a) + \alpha' \ln \frac{a}{\alpha'} + (1- \alpha') \ln\frac{1 - a}{1 -\alpha'} \Big)\,.$$
The above supremum is achieved at $x_{\theta,\alpha}$, as defined in Lemma 3.4 of \cite{HuGu}. 
We first prove the upper bound in \eqref{lim}.

\begin{eqnarray*}
 e^{N\Lambda_N(\bar{\theta}) }&=&  \E_e \E_{\bY} \Big[ \Big[ \exp \Big( 
 \frac{\beta N}{2}\sum_{i = 1}^{k} \theta_i \langle e_i, \bY_N e_i \rangle \Big) \Big] \Big] \\
 &=& \E_e \E_{\bY}\Big[ \Big[ \exp \Big( \frac{\beta}{2}
 \sum_{i = 1}^{k} \sum_{l=1,...,L \atop m = L+1,...,N} \theta_i  \sqrt{N} \Re(e_i(l) \bar{e}_i(m) X_{l,m} )\Big) \Big] \Big] \\
 & \leq &\E_e \Big[ \exp \Big( \frac{\beta N}{4} \sum_{l=1,...,L \atop m = L+1,...,N}\sum_{i,j = 1}^{k} \theta_i \theta_j \Re (e_i(l) \bar{e}_i(m)\bar{e}_j(l) e_j(m)) \Big) \Big]
 \end{eqnarray*}
 where we used that the entries are sharp sub-Gaussian. 
Now, let us call $e^{(1)}$ the vector of $\C^L$ whose coordinnates are the $L$ first coordinates of $e$ and $e^{(2)}$ the vector of $\C^M$ whose coordinates are the $M$ last of $e$. If we let $\psi^{(p)}_{l,m} = \langle e_l^{(p)},e_m^{(p)} \rangle$, the upper bound gives :
\begin{eqnarray*}
\Lambda_N(\bar{\theta}) &\leq&\frac{1}{N}\log \E_e \Big[ \exp \Big( \frac{\beta N}{2} \sum_{i,j = 1}^{k} \theta_i \theta_j  \psi^{(1)}_{i,j} \psi^{(2)}_{j,i} \Big) \Big]
\end{eqnarray*}
but since the $e_i$ are unitary and orthogonal , if we let $\Psi^{(p)} = ( \psi^{(p)}_{i,j})_{1\leq i,j \leq k }$ we have $\Psi^{(1)} + \Psi^{(2)} = I_{2k}$ and so $\psi^{(1)}_{i,j} \psi^{(2)}_{j,i} = \psi^{(1)}_{i,j} (\mathds{1}_{i = j} -  \bar{\psi}^{(1)}_{i,j})$. Furthermore the $\Psi^{(1)}$ is an element of a Jacobi ensemble as the following lemma states : 
\begin{lemma}\label{Jacobi}
The distribution of the matrix $\Psi^{(1)}$ when $ N > k$ is given by the following density for the Lebesgue measure on the set of symmetric/Hermitian matrices : 
\[ \frac{1}{Z} \det(\Psi^{(1)})^{\beta \frac{L -k +1}{2} - 1} \det( I_{k} - \Psi^{(1)})^{\beta \frac{M -k +1}{2} - 1} \mathds{1}_{ 0 \leq \Psi^{(1)} \leq I_k} d \Psi^{(1)} \]
\end{lemma}
\begin{proof}
	Let $U$ be a orthohogonal/unitary $N \times N$ Haar matrix, $U_1$ its $ L \times k$ top left block. Then $\Psi^{(1)}$ has the same law as $U_1^* U_1$. If we denote $\Pi$ the matrix $diag(\underbrace{1,...,1}_{\text{L times}},0,...0)$ and $\Pi '$ the matrix $diag(\underbrace{1,...,1}_{\text{k times}},0,...0)$, then $U_1^* U_1 = \Pi' U^* \Pi U \Pi '$. Then we can apply \cite[Theorem 2.2]{Col05} (up to  adapt this theorem  to the real case). 
\end{proof}

Therefore, using Laplace's method, we see that the distribution of $\Psi^{(1)}$ satisfies  a large deviations principle with rate function $I$ : 
\[ I(M) =\begin{cases} - \frac{\beta}{2} \Big[ \frac{1}{1+ \alpha} \ln \det ( M) + \frac{\alpha}{1 + \alpha} \ln \det (I_{k} - M) \Big]- Z   \text { if } 0 \leq M \leq I_{k}, \\
+ \infty \text{ otherwise.} \end{cases} \]
where $Z$ is such that $\min I = 0$. As a consequence, 
Varadhan's lemma implies that 
\[ \limsup \Lambda_N(\bar{\theta}) \leq  \Lambda(\bar{\theta}) \]
where :

\[ \Lambda(\bar{\theta}) = \sup_M [f(M) - I(M)] \]
with $f(M) =  \frac{\beta}{4} \sum_{i,j = 1}^{k} \theta_i \theta_j  M_{i,j} (I_{k} - M)_{j,i} $. We notice by taking $M=\alpha' I_k$ that
\[ Z \leq - \frac{k \beta}{2} \Big( \alpha \log \alpha' + (1- \alpha') \log (1- \alpha') \Big)\,.\] 
On the other hand, because $\det(M)\le \prod M_{ii}$ for any positive self-adjoint matrix $M$, 
\[ I(M) \geq  - \frac{\beta}{2} \Big[ \sum_{i = 1}^{k} \{\alpha' \ln(M_{i,i}) + (1- \alpha') \ln ( 1 - M_{i,i})\} \Big]+ Z \]
whereas $f(M)\ge  \frac{\beta}{4} \sum_{i= 1}^{k} \theta_i \theta_j  M_{i,i} (I_{k} - M)_{i,i}
$ since the off-diagonal terms are non-positive (because $M$ is symmetric and the $\theta_i$'s non-negative). We deduce (with $M_{i,i}=a_i$) that
\[
 \Lambda(\bar{\theta}) \le  \frac{\beta}{2} \sup_{  (a_i)_{i=1}^{k} \in ]0,1[^{k}} \sum_{i=1}^k  \Big( \theta_i^2 a_i (1- a_i) + \alpha' \ln \frac{a_i}{\alpha'} + (1- \alpha') \ln\frac{1 - a_i}{1 -\alpha'} \Big) = \sum_{i = 1}^k \Lambda(\theta_i)\]

To obtain the lower bound on  $\liminf_N \Lambda_N(\bar{\theta})$ as in \cite{HuGu}, it is enough to find a sequence of events $V_N^{\kappa}$ independent of $\Psi^{(1)}$ such that on these events $|e_i(l)| \leq C N^{-1/4 - \kappa}$ for some $\kappa>0$ and all $i$ and $l$ since then we will be in the regime where the sharp sub-Gaussian bound is also a lower bound. Note here that $\Psi^{(2)}$ is determined by $\Psi^{(1)}$, so we only condition on $\Psi^{(1)}$. To do that let us denote $U$ the $k\times L$ matrix with column vectors $(e^{(1)}_i, 1\le i\le k)$. Then 
$$U= (\Psi^{(1)} )^{1/2}V$$
and conditionally  to $\Psi^{(1)}$, $V=(v_1,\ldots, v_k)$ follows the uniform law on the set of $k$ orthonormal vectors on the sphere $\mathbb S_L$. We can then let $V_N^\kappa=\{\max_i \max_l |v_i(l)|\le N^{-1/4 - \kappa}\}$. On this set, $\max_i \max_l |e_i(l)| \leq C N^{-1/4 - \kappa}$ so that
\begin{eqnarray*}
 \Lambda_N(\bar{\theta}) &\ge &  \E_e \Big[1_{e \in V_N^\kappa} \exp \Big( \frac{\beta N}{4} (1+o(1))
  \sum_{l=1,...,L \atop m = L+1,...,N}\sum_{i,j = 1}^{k} \theta_i \theta_j \Re (e_i(l) \bar{e}_i(m)\bar{e}_j(l) e_j(m)) \Big) \Big]\\
  &=&  \E_e \Big[1_{e \in V_N^\kappa} \exp \Big( \frac{\beta N}{2} \sum_{i,j = 1}^{k} \theta_i \theta_j  \psi^{(1)}_{i,j} \psi^{(2)}_{j,i} \Big) \Big]
 \end{eqnarray*}
 where we expended the Laplace transform of the entries close to the origin. We finally notice that $V_N^\kappa$ is independent of $\Psi^{(1)}$ and with probability going to one. We can therefore apply the large deviations principle to deduce that
 \[ \limsup \Lambda_N(\bar{\theta}) \geq \sup_M [f(M) - I(M)] \,.\]
We finally conclude by taking $M$ diagonal that the above right hand side is bounded below by $\sum \Lambda(\theta_i)$, which completes the proof of \eqref{lim}.  
 To deduce the large deviations principle for the $k$ largest eigenvalues of Wishart matrices, we first obtain a large deviations upper bound by tilting the measure by the $k$-dimensional spherical integral. Because it factorizes as well as $\Lambda(\bar \theta)$ the upper bound has a rate function given by the sum of the rate functions for each outliers. To obtain the large deviations upper bound,  we tilt again  the measure by $\exp( \frac{\beta N}{2} \sum_{i = 1}^{k} \theta_i \langle e_i, \bY_N e_i \rangle)$ with $e_i \in V_N^{\epsilon}$. Under this tilted measure, we have the following expectations : 
 $$\E^{(e, \theta)}[\bY_N] = \sum_{i= 1}^k \theta_i \Big( e_i^{(1)}(e_i^{(2)})^* +e_i^{(2)}(e_i^{(1)})^* \Big) $$ ( where we identify $\C^L$ and $\C^M$ respectively with $\C^L \times \{0\}^M$ and $\{0\}^L \times\C^M $).  We can then write
\[ \bY_N = \tilde{\bY}_N + \sum_{i= 1}^k \theta_i \Big( e_i^{(1)}(e_i^{(2)})^* +e_i^{(2)}(e_i^{(1)})^* \Big) + o(1) \]
where $\tilde{\bY}_N$ has the same form as $\bY_N$ under the original measure. Then to identify the eigenvalues of $\bY_N$ outside the bulk of the limit measure we need to solve the following equation 
\[ \det\Big(  I_N + (\tilde{Y}_N - z)^{-1}\sum_{i= 1}^k \theta_i \big( e_i^{(1)}(e_i^{(2)})^* +e_i^{(2)}(e_i^{(1)})^* \big) \Big)= 0\]
Note that the above arguments also show that 
in $\Pp^{\bar{\theta}}$-probability $\Psi^{(1)}$ converges towards the diagonal matrix with entries $(x_{\theta_i,\alpha})_{1\le i\le k}$.  We also have 
local laws for $(z -\tilde{Y}_N)^{-1}$ under $\Pp^{\bar{\theta}}$. Therefore, if we denote $\tilde {\lambda_{+}} = \sqrt{(1+ \alpha)^{-1}\lambda_{+}}$ (which is the rightmost point of the support of the limit measure of $\bY_N$),
 the left hand side converges uniformly on any band $\{ z \in \C : \tilde \lambda_{+} + \epsilon \leq \Im z \leq A, |\Re z | \leq 1 \}$ toward :
\[ g : z \mapsto \prod_{i =1}^k \Big( 1 - \theta^2 z^2 x_{\theta_i,\alpha} (1 - x_{\theta_i,\alpha}) (1+ \alpha)^2  G_{MP(\alpha)}((1+ \alpha) z^2) G_{MP( 1/ \alpha)}((1+ \alpha) z^2) \Big) \]
where $MP(\alpha)$ is the Marchenko-Pastur distribution of parameter $\alpha$. Using the fact the these functions are holomorphic, we have the $k$ largest eigenvalue converges toward $z_{\theta_{1},\alpha} \geq  z_{\theta_{k},\alpha}$ where $z_{\theta,\alpha}$ is defined as the unique solution of 
\[ 1 - \theta^2 x_{\theta,\alpha} (1 - x_{\theta,\alpha}) (1+ \alpha)^2 z^2 G_{MP(\alpha)}((1+ \alpha) z^2) G_{MP( 1/ \alpha)}((1+ \alpha) z^2) = 0\]
on $] \tilde \lambda_+, + \infty [$ (see \cite{HuGu} for details).

\subsection{ Universality of the large deviations for the $k$ largest eigenvalues of Hermitian matrices with variance profiles  and sharp sub-Gaussian entries}
We consider in this section the setting of  Theorem \ref{thmuniv3}, which generalizes the previous subsection. We will proceed as in \cite{Hu} and we will first deal with the piecewise constant case with the supplementary technical assumption that the variance profile is non-negative. 

The main point is to prove the following estimate for the annealed spherical integral. 

\begin{lemma} Let $$\Lambda_N^\sigma(\bar{\theta}) =\frac{1}{N}\log  \E_{\bX,e}[\exp(N \sum_{i=1}^k \theta_i \langle e_i, \bX_N^\sigma e_i \rangle )]$$
Then, let $\sigma$ be piecewise constant and under the assumptions of Theorem \ref{thmuniv3}, for all $\theta_i\in\mathbb R^+$
$$\lim_{N\ra\infty}\Lambda^\sigma_N(\bar \theta)=\sum_{i=1}^k \Lambda^\sigma(\theta_i)$$
with, if $R_{ij}:=\sigma_{ij}^2$,
\[\Lambda^\sigma(\theta)= \frac{\beta}{2} \sup_{ \psi \in S } \Big[\frac{\theta^2}{2} \langle \psi_j, R \psi_j \rangle +  \sum_{i=1}^p \alpha_i \log \frac{\psi(j)}{\alpha_i} \Big] \]
\end{lemma}
Indeed, let us define for $e \in \Ss^{\beta N -1}$ and $j \in [ 1, p]$, $e^{(p)}$ the vector of $\C^{\alpha_i(N)}$ whose coordinates are the coordinates of $e$ whose indices lie in $I_N^i$. We then define for $j= 1,...,p$ the random matrix $\Psi^{(j)} = ( \langle e^{(j)}_i, e^{(j)}_j \rangle)$. Following the same computations as before and using the sharp sub-Gaussian character of the entries, we have :

\[ \Lambda^\sigma_N(\bar{\theta}) \leq \E_{e}\Big[\exp \Big ( \frac{\beta N }{4} \sum_{l,m=1}^k \sum_{i,j = 1}^p \theta_l \theta_m \Psi^{(i)}_{l,m} \bar{\Psi}^{(j)}_{l,m} \sigma_{i,j}^2 \Big) \Big] \]

 Notice that the $\Psi^{(j)}$ are Gram matrices (hence self-adjoint and positive) and that their sum is $I_k$.
There again we will use a slightly improved version of the Lemma \ref{Jacobi} to determine the distribution of the $\Psi^{(j)}$ : 
\begin{lemma}
	The joint distribution of the matrices $\Psi^{(1)},...,\Psi^{(p-1)}$ when $ \alpha_1(N),...,\alpha_p(N) > k$ is given by the following density for the Lebesgue measure on the set of symmetric/Hermitian matrices : 
	\[ \frac{1}{Z} \prod_{i=1}^{p-1} \Big( \mathds{1}_{ 0 \leq \Psi^{(i)}}\det(\Psi^{(i)})^{\beta \frac{\alpha_i(N) -k +1}{2} - 1} \Big)\det( I_{k} - \sum_{i=1}^{p-1} \Psi^{(i)})^{\beta \frac{ \alpha_p(N) -k +1}{2} - 1} \mathds{1}_{ \sum_{i=1}^{p-1} \Psi^{(i)} \leq I_k} \prod_{i=1}^{p-1} d \Psi^{(i)} \]
\end{lemma}
\begin{proof}
	Here we need an improved version of \cite[Theorem 2.2]{Col05} which states as follows.
		Let $U$ be a $N \times N$ Haar-distributed orthogonal or unitary matrix, $n_0 = 0 < n_1 < ...<   n_p = N$ a $p$-uplet of integers and for $i \in [1,p]$, $\tilde{\pi}_i$ the orthogonal projection on the vector span of the columns of $U$ with indices between $n_{i-1} + 1$ and $n_i$. Let $\pi$ be a constant projection of rank $k$. Then, if we identify $\pi S_N(\R) \pi$ (respectively $\pi H_N(\C) \pi$) to $S_k(\R)$ (respectively $H_k( \C)$), the joint distribution of $(M_1,...,M_{p-1}) = \pi \tilde{\pi}_1 \pi,...,\pi \tilde{\pi}_{p-1} \pi$ has the following density on $S_N(\R)^{p-1}$ ( resp. $H_N(\C)^{p-1}$) : 
		\[ \frac{1}{Z} \prod_{i=1}^{p-1} \Big( \mathds{1}_{ 0 \leq M_i}\det(M_i)^{\beta \frac{m_i -k +1}{2} - 1} \Big)\det( I_{k} - \sum_{i=1}^{p-1} M_i)^{\beta \frac{ m_p -k +1}{2} - 1} \mathds{1}_{ \sum_{i=1}^{p-1} M_i \leq I_k} \prod_{i=1}^{p-1} d M_i \]
		where $m_i = n_i - n_{i-1} $.

	The proof of this result is the same as the proof of \cite[Theorem 2.2]{Col05}. The difference is that one need to prove that $(\pi \tilde{\pi}_i \pi)_{1 \leq p -1}$ has the same law as $( \Sigma^{- 1/2} X_i \Sigma^{-1/2})$ where the $X_i$ are independent Gaussian Wishart of parameters $(k, m_i)$ and $\Sigma = X_1 + ...+ X_{p}$. 
Once we have this result, we take a Haar-distributed unitary matrix  $U$ and we denote $U_i$ the $ \alpha_i(N) \times k$ matrix extracted from $U$ by taking its $k$ first columns and its rows of indices in $I_N^i$. We denote $ \Pi' = \diag(\underbrace{1,...,1}_{ \text{ k times}}, 0, ..,0)$ and $\Pi_i$ the diagonal matrix with entries equal to $1$ for indices in $I_N^i$ and $0$ elsewhere. Then, since $(\Psi^{(i)})_{1 \leq p -1}$ has the same law as $( \Pi' U^* \Pi_i U \Pi')_{i \leq p-1}$, we can use the previous theorem. 
\end{proof}
We deduce from this explicit distribution of  the $p-1$-uplet $( \Psi^{(i)})_{1\leq i  \leq p}$ that it  follows a large deviations principle with rate function : 

\[ I( (M_i)_{1 \leq i \leq p-1})  = \begin{cases} 
 - \frac{\beta}{2} \Big [ \sum_{i=1}^{p} \alpha_i \ln \det(M_i)  \Big] - C \text{ if } \forall i\in [1,p], 0 \leq M_i \leq I_k, \text { and } \sum_{i=1}^p M_i = I_k \\
  + \infty \text{ otherwise.}
  \end{cases}
  \]
 Then we have using Varadhan's lemma : 
 
 \[ \limsup_N \Lambda^\sigma_N(\bar{\theta}) \leq \Lambda^\sigma( \bar{\theta}) \]
 where : 
 
 \[ \Lambda^\sigma(\bar{\theta}) = \sup_{(M_i)_{1\leq i \leq p}}[ f( (M_i)) - I( (M_i)) ] \]
 with $f((M_i)) = \frac{\beta}{4} \sum_{i,j=1}^{k} \theta_i \theta_j \langle M^{i,j}, R M^{i,j} \rangle$ where $R$ is the $p \times p$ matrix $( \sigma_{i,j}^2)$ and $M^{i,j}$ is the vector $(M_1(i,j),...,M_p(i,j))$. But, as before,  if $d(M)$ represents the diagonal matrix with entries $(M_{ii})_{1\le i\le k}$ we have that 
 $$I((M_i)) \geq I((d(M_i))) \mbox{ and for }i \neq j \langle M^{i,j}, R M^{i,j} \rangle \leq 0$$
 where the last inequality is due to Assumption \ref{Neg} which implies $\sum_{i=1}^p M_i=I$ and therefore that for $i\neq j$
 $\sum_{l=1}^p M^{i,j} =0$. Therefore we can again restrict the $\sup$ to diagonal matrices and it then decouples into 
 
 \[ \Lambda^\sigma(\bar{\theta}) = \frac{\beta}{2} \sup_{ (\psi_i)_{1 \leq i \leq k} \in S^{k}} \Big[ \sum_{j=1}^k \frac{\theta_i^2}{2} \langle \psi_j, R \psi_j \rangle + \sum_{j=1}^k \sum_{i=1}^p \alpha_i \log \frac{\psi_i(j)}{\alpha_i} \Big] = \sum_{j=1}^k \Lambda^\sigma( \theta_j) \]
 where $S := \{ \psi \in (\R^{+})^p : \psi_1 + ... + \psi_p =1 \}$. 
 In particular, since the function $\psi \mapsto \langle \psi, R \psi \rangle$ is concave on $S$ thanks to Assumption \ref{Neg}, the function optimized is strictly concave and thus is maximum at a unique  $\psi$. Furthermore $\psi_j$ only depends on $\theta_j$ so that we will denote $\psi_{j}= \psi_{\theta_j}$. Using again the strict concavity and the implicit function theorem, we have that the function $\theta \mapsto \psi_{\theta}$ is analytic in $\theta$. Furthermore, if we tilt our measure by $\E_{\bX}[ \exp( N \sum \theta_i \langle e_i, \bX_N^\sigma e_i \rangle )]$, the $\Psi^{(i)}$'s follow a large deviations principle and converges respectively toward $\diag( \psi_{\theta_1}(i),...,\psi_{\theta_k}(i))$. 
 For the lower bound we restrict the integral as in the preceding subsection to delocalized vectors with fixed $\Psi$ and conclude similarly. 
 
 To prove the large deviations principle, we first observe that the large deviations upper bound is direct after a tilt by spherical integrals and decoupling of the annealed spherical integrals.
  For the large deviations lower bound, we tilt  by $\exp( N \sum_{i=1}^k \theta_i \langle e_i, \bX_N^\sigma e_i \rangle)$. Under this tilted measure $\Pp^{e, \bar{\theta}}$, we have the following expectation $\E^{e, \bar{\theta}}[ \bX_N^\sigma] = \sum_{i=1}^k \theta_i \sum_{l,m=1}^p \sigma_{l,m}^2 e_i^{(l)} (e_i^{(m)})^*$. Using the BBP transition phenomenon, the local law for $\bX_N$ as in \cite[Lemma 5.6]{Hu} and the fact that the $\Psi^{(k)}$ converges in $\Pp^{\bar{\theta}}$ - probability, we have that the eigenvalues outside the bulk are asymptotically solution of the following equation in $z$ : 
 \[ \prod_{i=1}^k \det ( I_p - \theta_i R D(\theta_i,z)) = 0 \]
 where $D(\theta,z)$ is defined as in \cite[Section 5]{Hu}.
 To conclude, it suffices to prove that for any $z_1 > ...>z_k > r_{\sigma}$, there exists $\theta_1 \geq ...\geq \theta_k \geq 0$ such that $z_i$ is the unique solution of $\det(I_p - \theta R D(\theta,z)) =0$ on $[z_p, + \infty[$. We already know thanks again to the proof of the large deviations lower bound in \cite{Hu} that there is for every $z$, a $\theta$ such that $z$ is the largest solution. Let us prove that with Assumption \ref{Neg}, this solution is unique on $] r_{\sigma}, + \infty[$. First, one can notice that this assumption implies that the quadratic form whose matrix is $R$ has signature $(1, p -1)$ and so it is also true for the quadratic form whose matrix is $\sqrt{D(\theta,z)} R \sqrt{D(\theta,z)}$. Therefore, if we denote $\rho(\theta,z)$ the largest eigenvalue of $\sqrt{D(\theta,z)} R \sqrt{D(\theta,z)}$, the equation $\det(I_p - \theta_i R D(\theta_i,z)) =0$ is equivalent for $\theta >0$ and $z > r_{\sigma}$ to $\theta \rho(\theta,z) = 1$. Since $z \mapsto \rho(\theta,z)$ is strictly decreasing, the result is then proved. 
 
 For the continuous case, we can as in \cite[Section 6]{Hu} approximate our continuous variance profiles by piecewise constant ones. This approximation step is in fact easier than in the more general case of \cite{Hu} since if $\sigma$ satisfy Assumption \ref{Neg} then we can approximate $\bX_N^\sigma$ by the $\bX_N^{(p)}$ defined as follows : 

 \[ \bX_N^{(p)} = \sigma^{(p)}_N(i,j) \frac{X_{i,j}}{\sqrt{N}} \]
where 
$\sigma^{(p)}_N(i,j) = \sum_{k,l =1}^p \sigma^{(p)}_{k,l} \mathds{1}_{I^k_N \times I^l_N}(i,j)$ if 
 $I_N^1 = [0, N/p]$ and $I_N^i = ]N(i-1)/p, Ni/p]$ for $i = 2...p$ and 
\[ \sigma^{(p)}_{i,j} = \sqrt{ p^2 \int_{(i-1)/p}^{i/p} \int_{(j-1)/p}^{j/p} \sigma^2(x,y) dxdy }\,. \]
Since $\sigma$ satisfy Assumption \ref{Neg}, it is easy to check that $\sigma^ {(p)}$ also satisfies Assumption \ref{Neg} for all $p$ and therefore if we denote $\lambda^{(p),1}_N,...,\lambda^{(p),k}_N$ its $k$ largest eigenvalues, they satisfy a large deviations principle with rate function $I^{(p)}(x_1,...,x_p) = \sum I^ {(p)}(x_i)$. If we denote $\lambda^{1}_N,...,\lambda^{k}_N$ the $k$ largest eigenvalues of $\bX_N$, we have for all $i= 1,...,k$, $|\lambda^{i}_N - \lambda^{(p),i}_N| \leq ||\bX_N - \bX^{(p)}_N ||$. Using  \cite[Lemma 6.6]{Hu}, we have that $||\bX_N - \bX^{(p)}_N ||$ can be neglected  at exponential scale once $p$ is large enough. And using again \cite[Lemma 6.4 and Lemma 6.5]{Hu}, we have that the rate function converges toward the sum of the rate functions for one eigenvalue. Therefore, $\lambda^{1}_N,...,\lambda^{k}_N$ satisfy a large deviations principle and the rate function is the sum of the rate functions for one eigenvalue. 

\begin{remark}
Contrary to the Wigner case where we can see that asymptotically the positive and negative eigenvalues deviate independently from one another, this is not the case for matrices with variance profiles. An example is the linearization of a Wishart matrix where the negative eigenvalues are always exactly the opposite of the positive ones. 
\end{remark}

\subsection{ Large deviations for the $k$ largest eigenvalues  for the Gaussian ensembles with a $k$-dimensional perturbation}
We next prove proposition \ref{Wigper}. 
We first observe that the result is well known when $\bar\theta=0$, see e.g. Theorem \ref{thmuniv}.
We next remark that the joint law of the eigenvalues of $\bX_N^\theta$ is given by

$$d\mathbb P_N^\theta(\lambda)=\frac{1}{Z_N}\Delta(\lambda)^\beta\int \exp\{-\frac{\beta}{4}N\tr|UD(\lambda) U^*-\sum_{i=1}^{k}\theta_{i} e_{i}e_{i}^{*}|^2\}dU \prod_{1\le i\le N} d\lambda_i$$
where $U$ follows the Haar measure on the unitary group  (resp. the orthogonal group) when $\beta=2$ (resp, $\beta=1$).  $\Delta(\lambda)=\prod_{i<j} |x_i-x_j|$ is the Vandermonde determinant and $D(\lambda)$ is a diagonal matrix with entries given by $\lambda=(\lambda_1,\ldots,\lambda_N)$. Expanding the integral under the unitary (or orthogonal) group, we find that
$$d\mathbb P_N^\theta(\lambda)=\frac{1}{\tilde Z_N} \E_e[ e^{\frac{\beta}{2} N\sum_{i=1}^k \theta_i \langle e_i, D(\lambda) e\rangle }] d\mathbb P_N^0(\lambda)\,,$$
where $(e_1,\ldots,e_k)$ follows the uniform law on $k$ orthonormal vectors in dimension $N$. Hence the density is exactly given by the spherical integral. Using that Assumption \ref{ass} holds under $\mathbb P_N^0$ (see e.g \cite{GZ00}), we see that the empirical measure of $\lambda$ is close to the semi-circle law with overwhelming probability. Assume that $\theta_1\ge \theta_2\cdots\ge \theta_p\ge 0\ge\theta_{p+1}\cdots\ge\theta_k$. 
Then, on the set where the extreme eigenvalues $\lambda_N^N\ge \cdots\lambda_{N-p}^N$ and $\lambda^N_{1}\le \cdots\le\lambda^N_{k-p+1}$  are close to $x_{1}\ge x_{2}\ge\cdots\ge x_{p} \ge 2\ge -2\ge x_{-k+p}\ge \cdots \ge x_{-1}$, Theorem \ref{SIkdc} and Varadhan's Lemma give the result.

\subsection{ Large deviations for  $k$ extreme eigenvalues  for Gaussian Wishart matrices with a $k$-dimensional perturbation}
The proof of Proposition \ref{Wisper} is similar to the previous one. 
Again the proof is based on the explicit  joint law of $\lambda_{1}^{N,\gamma}\ge \lambda_{2}^{N,\gamma}\ge\cdots\lambda_{M}^{N,\gamma}$ given by the law on $(\mathbb R^+)^M$ 
$$d\mathbb P_{M,N}^{\bar \gamma}(d\lambda)=\frac{1}{Z_N}\Delta(\lambda)^\beta\int  e^{-\frac{\beta}{2}N\tr (U D(\lambda) U^* \Sigma^{-1})} dU \prod_{1\le i\le M}\lambda_i ^{\frac{\beta}{2}(N-M+1)}  d\lambda_i$$
Noticing that
$$\Sigma^{-1}=I+\sum_{i=1}^k \frac{\gamma_i}{1-\gamma_i} e_i e_i^*$$ we recognize again that the density with respect to the case $\gamma=0$ is given by a spherical integral. The result follows as for the Wigner case.

\section{Appendix}\label{app}
In this Appendix we investigate the continuity property of spherical integrals. First we need to prove the continuity of the deterministic limit itself :  
\begin{theorem}
	Let $d$ be a distance compatible with the weak topology on the set $\mathcal{P}(\R)$ and $ ||. ||$ any norm on $\R^{k+\ell}$, and for $M >0$, $\mathcal{K}_M$ the subset of $ E = \R^{k+\ell} \times (\R^{+})^k \times (\R^{-})^\ell \times \mathcal{P}(\R)$ defined by 
	\begin{eqnarray*}
		 \mathcal{K}_M := \{ (\bar{\lambda}, \bar{\theta}, \mu) \in E | M \geq \theta_1 \geq ...\geq \theta_k \geq 0 \geq \theta_{-\ell} \geq ... \geq \theta{-1} \geq  - M \\
		 M \geq \lambda_1 \geq ... \geq \lambda_k \geq r_{\mu} \geq l_{\mu} \geq \lambda_{-\ell} \geq ... \theta_{-1} \geq - M \}
		 \end{eqnarray*}
	 where $r_{\mu}$ and $l_{\mu}$ are respectively the rightmost and the leftmost point of the support of $\mu$. We endow $\mathcal{K}_M$ with the distance $D$ given by $D((\bar{\lambda}, \bar{\theta}, \mu), (\bar{\lambda}', \bar{\theta}', \mu') )= d(\mu, \mu') + || \bar{\lambda} - \bar{\lambda}'|| + ||\bar{\theta} - \bar{\theta}'||$. Then $\mathcal{K}_M$ is a compact set and the function $J$ 	 \[ J(\mu, \bar{\theta}, \bar{\lambda}) = \sum_{i= -\ell ,\neq 0}^k J( \mu, \theta_i,\lambda_i) \]
	 extends continuously  on $\mathcal{K}_M$. 
	 \end{theorem} 
\begin{proof}
	It is clear that we only need to prove the continuity of $(\theta, \lambda,\mu) \mapsto J( \mu,\theta,\lambda)$ where either $\theta \geq 0$ and $\lambda \geq r_{\mu}$ or $\theta \leq 0$ and $\lambda \leq l_{\mu}$. We assume without loss of generality that we are in the first case. Furthermore since $J(\mu, \theta, \lambda) = J( \theta* \mu, \theta \lambda, 1)$ we only need to prove the continuity for the first two arguments with the third being fixed equal to $1$. Let us take a sequence $( \mu_n, \lambda_n)$ such that $\forall n \in \N, l_{\mu} \geq - M, r_{\mu_n} \leq \lambda_n$ and $\lim \lambda_n = \lambda$ and $\lim_n \mu_n = \mu$. First, since $|J(\mu_n, \lambda_n,1) - J(\mu_n, \lambda + \epsilon,1)| \leq |\lambda_n - \lambda| + \epsilon$ for $n$ large enough so that $\lambda +\epsilon \geq r_{\mu_n}$, and $|J(\mu, \lambda,1) - J(\mu_n, \lambda + \epsilon,1)| \leq \epsilon$ we can restrict ourselves to proving $\lim J(\mu_n, \lambda + \epsilon,1) = J(\mu, \lambda + \epsilon,1)$. But, when we differentiate $J(\mu,\lambda,1)$ on the variable $\lambda$, we find
	\[ \frac{\partial}{\partial \lambda} J(\mu, \lambda,1) = 1 - \mathds{1}_{ [ G_{\mu}^{-1}(1), + \infty[} G_{\mu}(\lambda)- \mathds{1}_{ ] - \infty,  G_{\mu}^{-1}(1)[} G_{\mu}^{-1}(1) \]
On $[\lambda + \epsilon, + \infty [$, since $r_{\mu_n} \leq \lambda + \epsilon $ it is in fact easy to see that the weak convergence of $\mu_n$ imply the uniform convergence of $\partial/ \partial \lambda J( \mu_n, \lambda,1)$. The we conclude by choosing $\Lambda > \lambda + \epsilon$ so that $G_{\mu}( \Lambda) \leq 1/2$ so that $v(\mu_n,1, \Lambda) = \Lambda$ for $n$ large enough and then using the weak convergence and the fact that $ x \mapsto \log(\Lambda - x)$ is bounded on $[-M, \lambda + \epsilon]$, we have that $J( \mu_n, 1, \Lambda)$ converges toward $J( \mu,1, \Lambda)$.
	\end{proof}
With this continuity and the compactness of $\mathcal{K}_M$, we can prove the following theorem of uniform continuity, which generalizes \cite{Ma07} : 
 
\begin{theorem}\label{unifconv}
	Let $k,\ell \in \N$, $\bar{\theta} \in (\R^{-})^\ell \times(\R^{+})^k$ and 
	\[ J_N( X_N, \bar{\theta}) = \frac{1}{N} \log \E \Big[ \exp\Big( \frac{\beta N}{2} \sum_{i = -\ell, i \neq 0}^k \theta_i \langle e_i, \bX_N e_i \rangle \Big) \Big] \]
	
	 Let us denote $\lambda^N_{-1} \leq ... \leq \lambda^{N}_{-\ell}$ the smallest outliers of $\bX_N$ and $\lambda^N_{1} \geq ... \geq \lambda^{N}_{k}$ the largest outliers. We will denote this $k+\ell$-uplet $\bar{\lambda}_N$. Then for every $M >0$ and $\epsilon >0$, there is $N_0 \in \N$ so that
     for every $N \geq N_0$, for any  matrix $\bX_N$ such that $(\bar{\lambda}_N, \bar{\theta}, \mu_{\bX_N}) \in \mathcal{K}_M$
 \[   | J_N(\bX_N, \bar{\theta})  - J(\mu_{\bX_N}, \bar{\lambda}_N, \bar{\theta}) | \leq \epsilon \]
	\end{theorem}
\begin{proof} We first notice that in the proof of Proposition \ref{SIkdc} we approximated $J_N(\bX_N, \bar{\theta}) $ by $J_N(\bX_N^\varepsilon, \bar{\theta})$ with an error depending only on $\varepsilon$. Hence we may and shall replace in the above statement $\bX_N$ by $\bX_N^\delta$ for some small enough $\delta=\delta(\epsilon)$. $\bX_N^\delta$ has the same extreme eigenvalues than $\bX_N$ and otherwise eigenvalues $\lambda_{-\ell}^N+j\delta$ with multiplicity $\lfloor N\hat\mu^N([\lambda_{-\ell}^N+j\delta,\lambda_{-\ell}^N+(j+1)\delta])\rfloor$. Therefore, we see that $J_N(\bX_N, \bar{\theta}) $ is a function of the extreme eigenvalues and the empirical measure, hence a function on $\mathcal K_{N}$.  By the previous uniform approximation and the continuity of the limit, we deduce that it is uniformly continuous on $\mathcal{K}_{N}$, hence the result.
	\end{proof}

\bibliographystyle{amsplain}

\bibliography{bibAbel2}
\end{document}